\documentclass[a4paper,11pt, twoside, leqno]{amsart}

\usepackage{amsmath}
\usepackage{etoolbox}

\usepackage{amssymb}
\usepackage{ascmac}
\usepackage{amsthm}

\usepackage{enumerate}
\usepackage{mathrsfs}
\usepackage[pdftex]{graphicx}
\usepackage{float}
\usepackage{wrapfig}
\usepackage{layout}
\usepackage{cite}
\usepackage{braket}

\usepackage{amsfonts}
\usepackage[mathscr]{eucal}
\usepackage{xcolor}
\usepackage{enumitem}
\usepackage
[draft=false,setpagesize=false,pdfstartview=FitH,
colorlinks=true,linkcolor=blue,citecolor=magenta,pagebackref=true,bookmarks=false]
{hyperref}

\newcommand{\RN}{\mathbb{R}^N}

\newcommand{\e}{\varepsilon}
\newcommand{\rd}{\mathrm{d}}

\newcommand{\cL}{\mathcal{L}}

\newcommand{\la}{\left\langle}
\newcommand{\ra}{\right\rangle}

\def\ov#1{\overline{#1}}

\def\wt#1{\widetilde{#1}}

\def\ud#1{\underline{#1}}

\DeclareMathOperator{\supp}{supp}

\DeclareMathOperator{\sgn}{sgn}

\DeclareMathOperator{\card}{card}

\definecolor{RoyalBlue}{cmyk}{1, 0.5, 0, 0}

\newcommand{\ef}{\eqref}

\newcommand{\R}{\mathbb{R}}

\newcommand{\N}{\mathbb{N}}

\newcommand{\intN}{\int_{\R^N}}

\newcommand{\G}{\mathcal{G}}

\newcommand{\dis}{\displaystyle}

\usepackage[left=2.61cm,right=2.61cm,top=2.72cm,bottom=2.72cm]{geometry}
\usepackage{layout}

\usepackage{bm} 

\usepackage{aliascnt}

\usepackage[nameinlink]{cleveref}
 
\theoremstyle{plain}
\newtheorem{theorem}{Theorem}[section]
\newtheorem{lemma}[theorem]{Lemma}
\newtheorem{proposition}[theorem]{Proposition}

\theoremstyle{definition}

\theoremstyle{remark}
\newtheorem{remark}[theorem]{Remark}

\crefname{theorem}{Theorem}{Theorem}
\crefname{lemma}{Lemma}{Lemma}
\crefname{proposition}{Proposition}{Proposition}
\crefname{corollary}{Corollary}{Corollary}
\crefname{definition}{Definition}{Definition}
\crefname{remark}{Remark}{Remark}
\crefname{section}{Section}{sections}
\crefname{subsection}{Subsection}{subsections}

\numberwithin{equation}{section}

\providecommand{\keywords}[1]
{\small	\textit{Keywords and phrases:} #1}

\title[Locally superlinear Schr\"odinger equation]
{Existence and asymptotic behavior of positive solutions for a class of locally superlinear Schr\"odinger equation}

\author[S.~Adachi]{Shinji Adachi}
\author[N.~Ikoma]{Norihisa Ikoma}
\author[T.~Watanabe]{Tatsuya Watanabe}

\address[S.~Adachi]{\newline\indent
Department of Mathematical and Systems Engineering, 
\newline\indent 
Faculty of Engineering, Shizuoka University,
\newline\indent 
3-5-1 Johoku, Naka-ku, Hamamatsu, 432-8561, Japan}
\email{adachi@shizuoka.ac.jp}

\address[N.~Ikoma]{\newline\indent
Department of Mathematics, 
\newline\indent 
Faculty of Science and Technology, Keio University, 
\newline\indent 
Yagami Campus: 3-14-1 Hiyoshi, Kohoku-ku, Yokohama, Kanagawa, 223-8522, Japan}
\email{ikoma@math.keio.ac.jp}

\address[T.~Watanabe]{\newline\indent 
Department of Mathematics, 
\newline\indent 
Faculty of Science, Kyoto Sangyo University,
\newline\indent
Motoyama, Kamigamo, Kita-ku, Kyoto-City, 603-8555, Japan}
\email{tatsuw@cc.kyoto-su.ac.jp}

\thanks{}

\date{}

\begin{document}

\begin{abstract}
This paper treats the existence of positive solutions of $-\Delta u + V(x) u = \lambda f(u)$ in $\mathbb{R}^N$. 
Here $N \geq 1$, $\lambda > 0$ is a parameter and $f(u)$ satisfies conditions only in a neighborhood of $u=0$. 
We shall show the existence of positive solutions with potential of trapping type or $\mathcal{G}$-symmetric potential 
where $\mathcal{G} \subset O(N)$. 
Our results extend previous results \cite{AW2,CW05,DMS} as well as 
we also study the asymptotic behavior of a family $(u_\lambda)_{\lambda \geq \lambda_0}$ of positive solutions as $\lambda \to \infty$. 
\end{abstract}

\maketitle


\section{Introduction}
\label{section:1}

In this paper, we consider the following nonlinear Schr\"odinger equation:
	\begin{equation}\label{Plam}
		-\Delta u + V(x) u = \lambda f(u) \quad \text{in} \ \RN, 
		\quad u \in H^1(\RN),
		\tag{$P_\lambda$}
	\end{equation}
where $N \geq 1$ and $\lambda>0$ is a parameter.
Problem \ef{Plam} appears in the study of standing waves $\psi(t,x)=e^{-i \omega t} u(x)$ of the 
nonlinear Schr\"odinger equation:
\[
i \frac{\partial \psi}{\partial t} +\Delta \psi -\tilde{V}(x) \psi +\lambda f(\psi) =0  
\]
where $\tilde{V}(x)=V(x)-\omega$ and $f$ is assumed to satisfy 
$f(e^{i\theta}u)=e^{i \theta} f(u)$ for any $u$, $\theta \in \R$.
The existence of positive solutions for \ef{Plam} has been considered intensively 
since the work of Rabinowitz \cite{Ra92} 
and plays an important role in the study of the stability of standing waves. 
For more information, we refer to \cite{AS,BaWaWi05,Caz,DeGR,St08} and references therein. 
Although huge number of results has been obtained,
there still remains a gap of sufficient conditions on $f$ for the existence of positive solutions 
between \ef{Plam} and the autonomous problem (see \cite{BL83,BGK83}):
\[
-\Delta u =g(u) \quad \hbox{in} \   \R^N.
\]

	When we employ the variational method to find solutions of \eqref{Plam}, 
we usually require the nonlinear term $f$ to have at most $H^1$-critical growth at infinity. 
The existence result for \ef{Plam} with $V(x) \equiv 0$ on a bounded domain $\Omega$ in the $H^1$-supercritical case 
was studied in \cite{CW05}. More precisely, in \cite{CW05}, for a bounded domain $\Omega$ with smooth boundary, 
the authors considered the equation 
	\begin{equation}\label{eq:1.1}
		-\Delta u = \lambda f(u) \quad \text{in} \ \Omega, \quad u = 0 \quad \text{on} \ \partial \Omega
	\end{equation}
under the following conditions on $f(s)$: there exist $\delta > 0$ and $\beta,\gamma,\mu \in (2,2^\ast)$ where 
$2^\ast := 2N/(N-2)$ for $N \geq 3$ and $2^\ast := \infty$ for $N=1,2$ such that 
$f \in C^1(\R)$ satisfies 
	\begin{equation}\label{eq:1.2}
		\limsup_{|s| \to 0} \frac{f(s)}{|s|^{\gamma}} < \infty, \quad 
		\liminf_{|s| \to 0} \frac{F(s)}{|s|^\beta} > 0, \quad 
		0 < \mu F(s) \leq s f(s) \quad \text{for $0 < |s| < \delta$},
	\end{equation}
where $F(s) := \int_0^s f(t) \, \rd t$.  Assuming some further conditions on $\beta$ and $\gamma$, 
they proved the existence of positive, negative and sign-changing solution of \eqref{eq:1.1} for sufficiently large $\lambda$. 
See also \cite{CL05} for a multiplicity result with more restricted nonlinearity.

	On the other hand, \eqref{Plam} was studied in \cite{DMS} 
and the authors obtained the same result to \cite{CW05} for \eqref{Plam} by assuming the following conditions: 
$N \geq 3$, $f \in C^1(\R)$ satisfies \eqref{eq:1.2} and $V(x)$ satisfies (V1) and one of (V4'a) and (V4'b) below: 
\begin{enumerate}
	\item[(V1)] 
	$V \in C(\RN)$ and there exists $V_0>0$ such that $\dis \inf_{x \in \RN} V(x)=V_0$.
	\item[(V4'a)]
	For each $M>0$, $\cL^N \left( \Set{ x \in \RN | V(x) \leq M } \right) < \infty$ 
	where $\cL^N(A)$ stands for the $N$-dimensional Lebesgue measure of $A \subset \RN$. 
	\item[(V4'b)]
	$V^{-1} \in L^1(\RN)$. 
\end{enumerate}

	Recently, in \cite{AW2}, the existence of $\mathcal{G}$-symmetric solutions of \eqref{Plam} was established for $N \geq 2$. 
For the precise result in \cite{AW2}, see a paragraph after \cref{theorem:1.5}.

	This paper is motivated by \cite{CW05,DMS,AW2}. Our first aim is to relax the conditions in \eqref{eq:1.2}. 
Especially, we refine the last condition in \eqref{eq:1.2}, which is called the \emph{local Ambrosetti-Rabinowitz condition}. 
Here we concentrate on the existence of positive solutions of \eqref{Plam} and 
treat the case where $V$ is \emph{of trapping type}. 
Our second aim is to investigate the behavior of positive solutions $(u_\lambda)$ of \eqref{Plam} as $\lambda \to \infty$. 
To the best of the authors' knowledge, this problem is not studied yet. 
Thirdly, we also relax the conditions in \cite{AW2} on $f$ and $V$ for the existence of $\mathcal{G}$-symmetric solution. 
In particular, in \cite{AW2}, conditions on $\nabla V(x)$ are imposed, however, 
in this paper we do not require any condition on $\nabla V(x)$. 
In addition, we shall treat the case $N=1$ where we need some different analysis from the case $N \geq 2$. 
To the best of our knowledge, in case of $N=1$, 
the existence of even positive solutions of \eqref{Plam} with locally superlinear nonlinearity is not clear in the literature.

	In what follows, we state our main results.
For the potential $V$, through out this paper, we always assume (V1) above. 
As for the nonlinear term $f$, we impose the following conditions.
\begin{enumerate}
	\item[(f1)]
	$f \in C(\R)$ and $f(s) = 0$ for any $s \leq 0$. 
	\item[(f2)]
	$\displaystyle \lim_{s \to 0^+} \frac{f(s)}{s} = 0$. 
	\item[(f3)]
	There exists $ p \in \left( 2 , 2^\ast \right)$ such that 
	$\displaystyle \liminf_{s \to 0^+} \frac{F(s)}{s^p} > 0$. 
	\item[(f4)]
	There exists $s_1 > 0$ such that $\dis \frac{1}{2} f(s) s - F(s) > 0$ for all $s \in (0,s_1]$. 
\end{enumerate}

We first consider the case where the potential $V$ is of \emph{trapping type} and prove the following result:

\begin{theorem}\label{theorem:1.1}
	Suppose {\rm (V1)} and {\rm (f1)--(f4)}. 
	Furthermore, we assume either {\rm (i)} or {\rm (ii)} below holds: 
	\begin{enumerate}
	\item[{\rm (i)}] 
		\begin{enumerate}
		\item[{\rm (V2)}] 
		There exists $V_{\infty}>0$ such that $\dis \lim_{|x| \to \infty} V(x)=V_{\infty}$. 
		\item[{\rm (V3)}]
		$V(x) \le V_{\infty}$ for all $x\in \R^N$.
		\end{enumerate}
	\item[{\rm (ii)}] 
		\begin{enumerate}
		\item[{\rm (V4)}]
		There exists $r_0>0$ such that 
		for any $M>0$
			\[
				\lim_{|y| \to \infty} \cL^N \left( \Set{ x \in \RN | \, | x - y | \leq r_0, \ V(x) \leq M } \right) = 0.
			\]
		\end{enumerate}
	\end{enumerate}
	Then there exists $\lambda_0>0$ such that for $\lambda \geq \lambda_{0}$, 
	\eqref{Plam} admits a positive solution $u_\lambda$. 
\end{theorem}

	\begin{remark}\label{remark:1.2}
		\begin{enumerate}
		\item[(i)] 
		The conditions (f1)--(f4) are derived from \eqref{eq:1.2} and 
		an example which satisfies (f1)--(f4) but does not \eqref{eq:1.2} is 
			\[
				F(s) = \frac{s^2}{-\log s}, \quad f(s) = \frac{2s}{-\log s} + \frac{s}{(\log s)^2} 
				\quad \text{for $0< s \ll 1$}. 
			\]
		See \cite{AW2} for other typical examples of $f$ satisfying (f1)--(f4). 
		\item[(ii)]
		The condition (V4) can be found in \cite[Section 3]{BaWaWi05} and 
		it is used for compact embedding (see \cite[Lemma 3.1]{BaWaWi05} and \cref{lemma:3.7} below ). 
		Remark also that (V4'a) or (V4'b) implies (V4). 
		Hence, our result is a generalization of \cite{DMS} for the existence of positive solutions. 
		\item[(iii)]
		Using ideas of proof of \cref{theorem:1.1}, 
		we may prove the existence of positive solutions of \eqref{eq:1.1} under (f1)--(f4). 
		Hence, we may extend the result of \cite{CW05} regarding the existence of positive solutions. 
		See \cref{section:3.3}. 
		\end{enumerate}

	\end{remark}

	Next we investigate the asymptotic behavior of $(u_\lambda)_{\lambda \geq \lambda_{0}}$ as $\lambda \to \infty$. 
For this purpose, we introduce the following condition: 

\begin{enumerate}
	\item[{\rm (f5)}]
	There exists $p \in \left( 2 , 2^\ast \right) $ such that 
	$\displaystyle \lim_{s \to 0^+} \tfrac{f(s)}{s^{p-1}} = 1$. 
\end{enumerate}

		We remark that (f5) implies (f2), (f3) and (f4). 
Then we obtain the following result.

\begin{theorem}\label{theorem:1.3}
	Assume {\rm (V1)--(V3)}, {\rm (f1)} and {\rm (f5)}. 
	Let $(u_\lambda)_{\lambda \geq \lambda_0}$ be a family of positive solutions of \eqref{Plam} obtained in \cref{theorem:1.1}. 
		\begin{enumerate}
			\item[{\rm (a)}]
			If $V (x) \not \equiv V_\infty$, then for any $(\lambda_n)_{n=1}^\infty$ with $\lambda_n \to \infty$, 
				there exists a subsequence $(\lambda_{n_k})_{k=1}^\infty$ such that 
					\[
						v_k(x) := \lambda_{n_k}^{ 1/(p-2) } u_{\lambda_{n_k}} (x) \to v_\infty(x) \quad 
						\text{strongly in $H^1(\RN)$},
					\]
				where $v_\infty \in H^1(\RN)$ is a positive solution of 
					\[
						-\Delta v + V(x) v = v^{p-1} \quad \text{in} \ \RN.
					\]
				\item[{\rm (b)}]
				If $V(x) \equiv V_\infty$, then there exists $(x_\lambda)_{\lambda \geq 1} \subset \RN$ such that
				$ \| v_\lambda (\cdot + x_\lambda) - \omega_0 \|_{H^1(\RN)} \to 0$ as $\lambda \to \infty$ 
				where $v_\lambda (x) := \lambda^{1/(p-2)} u_{\lambda} (x)$ and $\omega_0 \in H^1(\R^N)$ is the unique positive radial solution of 
					\begin{equation}\label{eq:1.3}
						-\Delta u + V_\infty u = u^{p-1} \quad \text{in} \ \RN.
					\end{equation}
			\end{enumerate}
		Furthermore, assertion {\rm (a)} holds true if we replace {\rm (V2)} and {\rm (V3)} by {\rm (V4)}. 
	\end{theorem}

	\begin{remark}\label{remark:1.4}
As in \cref{theorem:1.1}, we may also prove similar results for \eqref{eq:1.1}. 
	\end{remark}

	Next we study the existence of $\mathcal{G}$-symmetric solution of \ef{Plam}. 
Let us denote by $O(N)$ the orthogonal group in $\R^N$ and suppose 
	\begin{equation}\label{g1}
		\left\{\begin{aligned}
			& \text{$\G \subset  O(N)$ is a finite subgroup}, 
			\\
			&\text{for any $x \in S^{N-1} := \Set{ x \in \RN | \, |x| = 1 } $ there exists $g \in \G$ such that $gx \neq x$}.
		\end{aligned}\right.
		\tag{g1}
	\end{equation}
Set 
\begin{equation} \label{eq:1.4}
	k_0 := \min_{x\in S^{N-1}} \card \Set{gx | g\in \G} \geq 2
\end{equation}
and choose $x_0\in S^{N-1}$ such that $\card \Set{ g x_0 | g \in \G } = k_0$. 
Put also 
\begin{equation} \label{eq:1.5}
	\Set{ gx_0 | g\in \G } = \Set{ e_1, \dots ,e_{k_0} } \subset S^{N-1} \ \hbox{and} \
	\alpha_0 := \min_{i \ne j} |e_i-e_j| \in (0,2]. 
\end{equation}
For $V(x)$, in addition to (V1), we assume 
\begin{enumerate}
\item[\rm(V5)] $V(gx)=V(x)$ for all $x \in \R^N$ and $g \in \G$.

\smallskip
\item[\rm(V6)] There exist $\alpha \in (\alpha_0 \sqrt{V_{\infty}}, \infty)$ and $\kappa>0$ such that 
\[
V(x) \leq V_\infty + \kappa e^{ - \alpha |x| } \quad \hbox{for all} \ x \in \RN.
\]
\end{enumerate}
For the nonlinearity $f$, we impose slightly stronger assumptions:
\begin{enumerate}
	\item[(f2')]
	There exists $\nu>0$ such that $\displaystyle \lim_{s \to 0+} \frac{f(s)}{s^{1+\nu}}=0$.
	\item[(f6)] 
	There exists $s_2>0$ such that $\displaystyle \frac{f(s)}{s}$ is nondecreasing in $(0,s_2)$. 
\end{enumerate}

Under these assumptions, we can show the following result.

\begin{theorem} \label{theorem:1.5}
Suppose $N \ge 2$, \eqref{g1}, {\rm (V1), (V2), (V5), (V6), (f1), (f2'), (f3), (f4)} and {\rm (f6)}.
Then there exists $\lambda_0>0$ such that \eqref{Plam} admits a $\mathcal{G}$-symmetric positive solution $u_{\lambda}$
for $\lambda \ge \lambda_0$.
\end{theorem}

	In \cite{AW2}, the existence of $\G$-symmetric solution of \eqref{Plam} was obtained 
under $N \geq 3$, \eqref{g1}, (V1), (V2), (V5), (V6), (f1), (f6), 
	\begin{equation}\label{eq:1.6}
		\lim_{t \to 0^+} \frac{f(t)}{t^{\gamma-1}} = 0, \quad 0 < \liminf_{t \to 0^+} \frac{F(t)}{t^\beta}
	\end{equation}
for some $2 < \gamma \leq \beta < 2^\ast$ and one of the following conditions: 
	\begin{enumerate}
	\item[(i)]
		There exists $\alpha \in (0,1)$ such that 
		\[
		 \| (\nabla V(x) \cdot x )_+ \|_{L^{N/2} (\RN)} \leq 2 \alpha S_N
		 \]
		  where $a_+ := \max \{0,a\}$ and 
		$S_N$ denotes the best constant of Sobolev's inequality: $\| u \|_{L^{2^\ast} (\RN)} \leq C \| \nabla u \|_{L^2(\RN)}$. 
	\item[(ii)]
		There exists $\alpha \in (0,1)$ such that 
		\[
			\left(  \nabla V (x) \cdot x \right)_+ 
			\leq \frac{(N-2)^2}{2|x|^2} \alpha \quad 
			\text{for all $x \in \RN \setminus\{0\}$}.
		\]
	\end{enumerate}
\cite{AW2} also dealt with the case $N=2$ and 
the existence result was shown by assuming 
\eqref{g1}, (V1), (V2), (V5), (V6), (f1), (f6) and \eqref{eq:1.2} with $\gamma \leq \beta$. 
Notice that \eqref{eq:1.6} and (f6) (or \eqref{eq:1.2} and (f6)) yield 
the existence of $\wt{s} > 0$ such that $f(s) > 0$ on $(0, \wt{s} ]$ and for every $s \in (0,\wt{s}]$, 
	\[
		f(s) s - 2 F(s) = f(s) s - 2 \int_0^s \frac{f(\tau)}{\tau} \tau \, \rd \tau > f(s) s - 2 \frac{f(s)}{s} \int_0^s \tau \, \rd \tau = 0.
	\]
Therefore, (f1), (f2'), (f3), (f4) and (f6) in \cref{theorem:1.5} are weaker than those in \cite{AW2}. 
In addition, we succeed to drop the conditions on $\nabla V(x)$. 
Thus, \cref{theorem:1.5} extends \cite{AW2}.

Finally we investigate the existence of even positive solutions when $N=1$. 
As an application of \cref{theorem:1.3}, we prove the following result.

	\begin{theorem}\label{theorem:1.6}
		Suppose $N=1$, {\rm (V1), {\rm (V2)}, (f1), (f5)} and {\rm (f6)}. In addition, we assume that $V(x)$ satisfies 
			\begin{enumerate}
				\item[{\rm (V5')}]
				For all $x \in \R$, $V(-x) = V(x)$. 
				\item[{\rm (V6')}]
				There exist $\alpha \in \left( 2 \sqrt{V_\infty} , \infty \right)$ and $\kappa > 0$ such that 
					\[
						V(x) \leq V_\infty + \kappa e^{ - \alpha |x| } \quad \text{for all $x \in \R$}.
					\]
			\end{enumerate}
		Then there exist $\lambda_0=\lambda_{0}(V,f) > 0$, $C_0=C_0(V_\infty,p) > 0$  
		and $\wt{C}_0 = \wt{C}_0( V_\infty, p ) > 0$ such that 
		if the inequalities $\lambda \geq \lambda_0$ and 
			\begin{equation}\label{eq:1.7}
				- C_0 + \frac{\wt{C}_0 \kappa \alpha }{\alpha^2 - 4 V_\infty} < 0
			\end{equation}
		hold, then \eqref{Plam} admits an even positive solution $u_\lambda$. 
	\end{theorem}

\begin{remark}
It is worth noting that the constants $C_0$ and $\wt{C}_0$ in \cref{theorem:1.6} 
depend only on $V_\infty$ and $p$, not on $V$ itself where $p$ appears in (f5). 
Hence, after fixing a constant $V_\infty$ and a function $f$ satisfying (f1), (f5) and (f6), 
we may verify \eqref{eq:1.7} provided either $\alpha \in (2\sqrt{V}_\infty, \infty  )$ is large or $\kappa > 0$ is small. 
Therefore, for given $V_\infty$ and $f$, and 
for any $\alpha \in (2\sqrt{V_\infty}, \infty)$ and $\kappa > 0$ satisfying \eqref{eq:1.7}, 
we may find $V(x)$ which enjoys (V1), (V2), (V5') and (V6'), and \cref{theorem:1.6} can be applied. 
\end{remark}

	Here we explain an outline and ideas of proofs. 
Since the nonlinear term $f$ may have the $H^1$-supercritical growth, 
following \cite{AW2, CW05, DMS}, 
we modify $f$ and write $g$ for the modification of $f$. 
Then we aim to show the existence of a family $(u_\lambda)_\lambda$ of positive solutions of the modified problem
	\[
		-\Delta u + V(x) u = \lambda g (u) \quad \text{in} \ \RN
	\]
satisfying $\| u_\lambda \|_{L^\infty(\RN)} \to 0$ as $\lambda \to \infty$. 
Since we do not suppose the local Ambrosetti--Rabinowitz condition, the boundedness of Palais--Smale sequences 
is an issue. For this issue, in \cite{AW2}, the monotonicity trick (see \cite{Je99}) was employed and 
some conditions on $\nabla V$ were necessary to exploit the Pohozaev identity. 
In this paper, we shall use Palais--Smale--Cerami sequences instead of Palais--Smale sequences 
and our approach is similar to \cite{AS} for the existence of positive solutions. 
See the comments after \eqref{eq:3.3} for the relation between \cite{AS} and our argument. 
After the boundedness of Palais--Smale--Cerami sequences is established, 
we show the existence of positive solution via the concentration compactness argument due to Lions \cite{Li84a,Li84b}. 
In this paper, we will use a version in \cite[Chapter 8]{Wi96}.

	In order to prove $\| u_\lambda \|_{L^\infty(\RN)} \to 0$ as $\lambda \to \infty$, 
in \cite{AW2,CW05,DMS}, we point out that 
the local Ambrosetti--Rabinowitz condition or the Pohozaev identity was also used. 
Therefore, to obtain $\| u_\lambda \|_{L^\infty(\RN)} \to 0$ without using these conditions, 
some ideas are necessary. Here we use a slightly different modification of $f$ to those in \cite{AW2,CW05,DMS} 
and require $g$ to satisfy the following estimate: each $\tau > 0$ there exists a $c_\tau > 0$ such that 
\begin{equation} \label{eq:1.8}
\frac{1}{2} g (s) s -G(s) \ge c_{\tau} s^p \quad \hbox{for all} \ s\in [\tau,\infty) \quad \text{where} \ G(s) := \int_0^s g(t) \, \rd t. 
\end{equation}
See \cref{lemma:2.1} for more precise information. 
We remark that \eqref{eq:1.8} is also useful to obtain the boundedness of Palais--Smale--Cerami sequences. 
To prove $\| u_\lambda \|_{L^\infty (\RN)} \to 0$ through \eqref{eq:1.8}, 
we focus on the superlevel set of $u_\lambda$ and 
first derive the decay estimate in $\lambda$ of $\| (u_\lambda - \tau)_+ \|_{L^{p}(\RN)}$ 
for each fixed $\tau > 0$. 
After that, we may get the decay estimate of $\| (u_\lambda - \tau)_+ \|_{H^1(\RN)}$ and 
$\| (u_\lambda - \tau)_+ \|_{L^\infty(\RN)}$, which yields $\| u_\lambda \|_{L^\infty (\RN)} \to 0$ 
due to the arbitrariness of $\tau>0$. These will be done in \cref{proposition:3.6}.

	Finally we emphasize that the case $N=1$ is different from the case $N \geq 2$. 
Indeed, according to \cite{GS,Sa}, \ef{Plam} may not have any positive solutions in some case. 
We will see the difference between $N =1$ and $N \geq 2$ from technical point of view in the end of \cref{section:5}. 
We also mention that a sufficient condition for the existence of even positive solutions 
becomes complicated and implicit. For instance, the existence of nontrivial solution was studied in \cite{ES09,EGS} 
and a key assumption in \cite{ES09,EGS} is \cite[($\mathrm{A}_2$)]{EGS}. 
In this paper, we will perform a refined interaction estimate when $N=1$ to prove a strict inequality 
which is similar to \cite[($\mathrm{A}_2$)]{EGS} (see \cref{proposition:6.2} below). 
In this procedure, the asymptotic behavior of positive solutions $u_{\lambda}$ of \ef{Plam}
as $\lambda \to \infty$, which is obtained in \cref{theorem:1.3}, plays an important role.

This paper is organized as follows. 
In \cref{section:2}, we introduce a modified nonlinearity which possesses several global properties.
\cref{section:3} is devoted to the study of the existence of a positive solution of \ef{Plam}
and the proof of \cref{theorem:1.1}.
In \cref{section:4}, we investigate the asymptotic behavior of positive solutions of \ef{Plam} as $\lambda \to \infty$
and give the proof of \cref{theorem:1.3}.
\cref{section:5} deals with $\mathcal{G}$-invariant case, 
and \cref{section:6} is dedicated to one-dimensional case.

%
%
%


\section{Modification of the nonlinear term}
\label{section:2}

In this section, we introduce a modified nonlinearity $g_{m_0}$ which satisfies several global properties.
First by (f3) and (f4), we can find $s_0>0$ and $c_0>0$ such that 
\begin{equation}\label{eq:2.1}
	F(s) \geq c_0 s^{p}, \quad \frac{1}{2} f(s) s - F(s) > 0  
	\quad \text{for all $s \in (0,4s_0]$}.
\end{equation}
We remark that \eqref{eq:2.1} implies 
\begin{equation}\label{eq:2.2}
	f(s) > 0 \quad \text{for all $s \in (0,4s_0]$}. 
\end{equation}
Next, we fix $\xi_1,\eta_1 \in C^\infty(\R)$ so that 
\begin{equation}\label{eq:2.3}
	\begin{aligned}
		\xi_1 & \equiv 1 & &\text{on} \ [0,3s_0], 
		& \xi_1 & \equiv 0 & &\text{on} \ [4s_0,\infty), 
		& \xi_1' & \leq 0 & &\text{on} \ [0,\infty),
		\\
		\eta_1 & \equiv  0 & &\text{on} \ [0,s_0],
		& \eta_1 & \equiv 1 & &\text{on} \ [2s_0,\infty), 
		& \eta_1' & \geq 0 & & \text{in} \ [0,\infty).
	\end{aligned}
\end{equation}
For $m > 0$, we define $g_m(s)$ by 
	\begin{equation}\label{eq:2.4}
		g_m(s) := \xi_1(s) f(s) + m \eta_1(s) s^{p-1} \quad \text{for $s \geq 0$}, \quad 
		g_m(s) := 0 \quad \text{for $s \leq 0$}
	\end{equation}
and set $G_m(s) := \int_0^s g_m(t) \, \rd t$. 
As mentioned in the introduction, (f5) implies (f3) and (f4),
and hence, \eqref{eq:2.1} and \eqref{eq:2.2} still hold under (f5) only.

\begin{lemma} \label{lemma:2.1}
Assume {\rm (f1)--(f4)}. Then there exists $m_0 > 0$ such that 
$g_{m_0}$ satisfies the following properties: 
\begin{enumerate}
\item[{\rm (i)}] 
For any $\tau \in (0,s_0)$, there exists a constant $c_\tau > 0$ such that 
\begin{alignat}{2} 
		&c_\tau s^p \leq \frac{1}{2}g_{m_0} (s) s - G_{m_0} (s) 
		&
		\quad &\text{for all $s \in [\tau,\infty)$}, \label{alignat:2.5}
		\\
		& 0 < \frac{1}{2} g_{m_0} (s) s - G_{m_0} (s)
		& &\text{for all $s \in (0,\infty)$}. \label{alignat:2.6}
\end{alignat}

\item[{\rm (ii)}]
There exists a constant $\wt{c}>0$ such that 
\begin{equation}\label{eq:2.7}
	\wt{c} s^p \leq G_{m_0} (s) \quad \text{for all $s \in [0,\infty)$}.
\end{equation}

\item[{\rm (iii)}]
Assume further {\rm (f5)}. Then by taking $s_0>0$ smaller if necessary, 
there exist $\mu \in (2,p)$ and $0<C_1 \leq C_2 < \infty$ such that
	\begin{alignat}{2}
		&0 < \mu G_{m_0 } (s) \leq g_{ m_0 } (s) s \quad & & \text{for all $s \in (0,\infty)$},
		\label{alignat:2.8}
		\\
		&C_1 s^{p-1} \leq g_{m_0} (s) \leq C_2 s^{p-1} \quad & & \text{for all $s \in (0,\infty)$}.
		\label{alignat:2.9}
	\end{alignat}

\item[{\rm (iv)}]
If {\rm (f6)} holds, then $s^{-1}g_{m_0}(s)$ is nondecreasing in $(0,\infty)$.
\end{enumerate}

\end{lemma}

\begin{proof}
We prove (i). First we claim that there exists $m_0>0$ such that 
\begin{equation}\label{eq:2.10}
	m \geq m_0, \ s \geq 3s_0 \quad \Rightarrow \quad \frac{1}{2} g_m (s) s - G_m (s) \geq s^{p}.
\end{equation}
Indeed when $s \ge 3s_0$, one finds from \ef{eq:2.3} that
	\begin{equation}\label{eq:2.11}
		\sup_{s \ge 0}  \int_0^s \xi_1(t) f(t) \, \rd t  =:M_0 < \infty.
	\end{equation}
Moreover writing $\wt{g} (s) := \eta_1(s) s^{p-1}$ and $\wt{G}(s) := \int_0^s \wt{g}(t) \, \rd t$, 
by \ef{eq:2.3} and $p>2$, we also have
	\begin{equation}\label{eq:2.12}
	0 \leq 2 \wt{G} (s) \leq  p \wt{G}(s) \leq p \eta_1(s) \int_0^s t^{p-1} \, \rd t = \eta_1(s) s^{p} = \wt{g}(s) s
	\quad \text{for any $s \in [0,\infty)$}.
\end{equation}
Thus, from \eqref{eq:2.2}, \eqref{eq:2.3} and \eqref{eq:2.12}, it follows that for $s \ge 3s_0$, 
\[
	\begin{aligned}
		\frac{1}{2}g_m (s) s - G_m (s) 
		&= 
		\left\{ \frac{1}{2} \xi_1(s) f(s) s - \int_0^s \xi_1(t) f(t) \, \rd t \right\} 
		+ m \left\{ \frac{1}{2} \wt{g} (s) s - \wt{G}(s) \right\}
		\\
		& \geq 
		- M_0 + m \left( \frac{1}{2} - \frac{1}{p} \right) \wt{g} (s) s 
		\\
		& = 
		- M_0 + m \left( \frac{1}{2} - \frac{1}{p} \right) s^p,
	\end{aligned}
\]
from which we deduce \ef{eq:2.10}.

	Next we show \ef{alignat:2.5} and \ef{alignat:2.6}. 
We choose $\tau \in (0,s_0)$ arbitrarily and suppose that $\tau \le s \le 3s_0$. 
Then from $g_m (s) = f(s) + m \wt{g}(s)$, \eqref{eq:2.1} and \eqref{eq:2.12}, it follows that 
\[
	\frac{1}{2} g_m (s) s - G_m (s) 
	= 
	\frac{1}{2} f(s) s + \frac{m}{2} \wt{g} (s) s - F(s) - m \wt{G}(s) 
	\geq \frac{1}{2} f(s) s - F(s) > 0.
\]
Hence, we can find a constant $c_\tau > 0$ such that
\begin{equation*}
	c_\tau s^p \leq \frac{1}{2} g_m (s) s - G_m (s) 
	\quad \text{for any $m > 0$ and $s \in [\tau,3s_0]$}.
\end{equation*}
Then \ef{eq:2.10} and \ef{eq:2.12} yield \eqref{alignat:2.5}. 
Moreover since $g_{m_0} \equiv f$ on $[0,\tau]$, 
\ef{alignat:2.6} follows from (f4).

Next, we prove (ii). By \eqref{eq:2.1} and \eqref{eq:2.3},  
we have $c_0s^p \leq F(s) \leq G_{m_0} (s)$ for all $s \in [0,3s_0]$. 
On the other hand if $3s_0 < s$, choosing $\wt{c} \in (0,c_0]$ sufficiently small, 
we see from $\eta_1(s) = 1$ and $\xi_1(s) f(s) \geq 0$ that 
\[
	\frac{\rd}{\rd s} \left\{ G_{m_0} (s) - \wt{c} s^p \right\} \geq m_0 s^{p-1} - p \wt{c} s^{p-1} > 0, 
	\quad G_{m_0} (3s_0) - \wt{c} (3s_0)^p > 0.
\]
Thus, it follows that 
$G_{m_0}(s) \geq \wt{c} s^p$ for all $s \in [3s_0,\infty)$ and hence (ii) holds.

We show (iii). By (f5), one finds that
\[
	\lim_{s \to 0^+ } \frac{f(s)}{s^{p-1}} = 1, \quad 
	\lim_{s \to 0^+}  \frac{F(s)}{s^p} = \lim_{s \to 0^+} \frac{1}{s^p} \int_0^s \frac{f(t)}{t^{p-1}} t^{p-1} \, \rd t 
	= \frac{1}{p}.
\]
For $\mu \in (2,p)$, by taking $s_0$ smaller if necessary, we have 
\[
	0<\mu \frac{F(s)}{s^p} \leq \frac{f(s)}{s^{p-1}} \quad \text{for each $s \in (0,4s_0]$},
\]
and hence 
\[
	0 < \mu F(s) \leq f(s) s \quad \text{for any $s \in (0,4s_0]$}.
\]
Since $\mu \in (2,p)$, by \eqref{eq:2.11}, \eqref{eq:2.12} and replacing $m_0$ by larger one if necessary, 
we see for $s \in (0,3s_0]$
	\[
		g_{m_0}(s) s - \mu G_{m_0} (s) = f(s) s - \mu F(s) + m_0 \left( \wt{g}(s) s - \mu \wt{G}(s) \right) \geq 0
	\]
and for $s \in (3s_0,\infty)$ 
	\[
		g_{m_0} (s) s - \mu G_{m_0} (s) \geq - \mu M_0 + m_0 \left( 1 - \frac{\mu}{p} \right) \wt{g} (s) s \geq 0.
	\]
Moreover by \eqref{eq:2.3}, (f5) and the definition of $g_{m_0}(s)$, 
we can verify \eqref{alignat:2.9} easily.

	Finally, we prove (iv). 
Without loss of generality, we may assume $4s_0 \leq s_2$. 
By (f6) and \eqref{eq:2.3}, it suffices to show that 
$s^{-1} g_{m_0}(s)$ is nondecreasing in $[3s_0,4s_0]$. 
Let $3s_0 \leq t_1 < t_2 \leq 4s_0$ and set 
	\[
		M_1 := \max_{ 3s_0 \leq t \leq 4s_0 } \left| \xi'(t) \frac{f(t)}{t} \right| < \infty.
	\]
Then \eqref{eq:2.3} and (f6) imply 
	\[
		\begin{aligned}
			\frac{g_{m_0} (t_2) }{t_2} - \frac{g_{m_0}(t_1)}{t_1}
			&= \xi(t_2) \left( \frac{f(t_2)}{t_2} - \frac{f(t_1)}{t_1} \right) 
			+ \left( \xi(t_2) - \xi(t_1) \right) \frac{f(t_1)}{t_1} + m_0 \left( t_2^{p-2} - t_1^{p-2} \right) 
			\\
			&\geq \int_{t_1}^{t_2} \xi'(\tau) \frac{f(t_1)}{t_1} + (p-2) m_0 \tau^{p-3} \, \rd \tau 
			 \\
			 &\geq \int_{t_1}^{t_2} (p-2) m_0 \tau^{p-3} - M_1 \, \rd \tau \geq 0
		\end{aligned}
	\]
provided $m_0>0$ is large. Hence, (iv) holds. 
\end{proof}

	\begin{remark}\label{remark:2.2}
		From the proof of \cref{lemma:2.1} (i) under (f1)--(f4), for $\mu \in (2,p)$, 
		it is also possible to verify $  0 < \mu G_{m_0} (s) \leq g_{m_0} (s) s$ for all $s \in [4s_0,\infty)$ 
		if $m_0$ is sufficiently large. 
	\end{remark}

\section{Existence of a positive solution of \ef{Plam}: proof of \cref{theorem:1.1}}
\label{section:3}

	In this section, we shall prove \cref{theorem:1.1}. 
Here we mainly treat case (i), namely, we assume (V2) and (V3) 
since case (ii) is simpler than case (i).

\subsection{Proof of \cref{theorem:1.1} under (i)}
\label{section:3.1}

	Through out this subsection, we always assume (V1) and (V2). 
Hereafter, let $m_0 > 0$ be a constant in \cref{lemma:2.1} 
and consider the following auxiliary problem:
	\begin{equation}\label{eq:3.1}
		-\Delta u + V(x) u = \lambda g_{m_0} (u) \quad \text{in} \ \RN.
	\end{equation}
If a solution $u$ of \eqref{eq:3.1} satisfies $\| u \|_{L^\infty(\RN)} \leq s_0$, 
then by the definition of $g_{m_0}$ in \ef{eq:2.4}, 
$u$ is also a solution of \eqref{Plam}. 
Therefore, we aim to find such a solution of \ef{eq:3.1}.

	To find solutions of \eqref{eq:3.1}, we define $I_\lambda(u)$ by 
	\begin{equation}\label{eq:3.2}
		I_\lambda(u) := \frac{1}{2} \int_{\RN} |\nabla u|^2 + V(x) u^2 \,\rd x - \lambda \intN G_{m_0} (u) \, \rd x.
	\end{equation}
From \eqref{eq:2.4} and (f2), it is clear that $I_\lambda \in C^1(H^1(\RN), \R)$ and 
critical points of $I_\lambda$ are solutions of \eqref{eq:3.1}.

	For later use, we also define the scaled functional $J_\lambda(u)$ by 
	\[
		J_\lambda (v) := \frac{1}{2} \int_{\RN} |\nabla v|^2 + V(x) v^2 \,\rd x  
		- \lambda^{p/(p-2)} \intN G_{m_0} \left( \lambda^{- 1/(p-2) } v \right) \,\rd x \in C^1 \big( H^1(\RN), \R \big). 
	\]
Then one finds that 
	\begin{equation}\label{eq:3.3}
		J_\lambda \left( v \right) = \lambda^{2/(p-2)} I_\lambda \left( \lambda^{ -1/(p-2) } v \right).
	\end{equation}
Even though it is possible to apply the argument of \cite[Theorem 1.1]{AS} due to \cref{remark:2.2}, 
the case where $V$ satisfies (V1)--(V3) is not explicitly treated in \cite{AS} 
(notice that \cite[Class 3]{AS} is slightly different from our case). 
Therefore, for the sake of clarity and readers' convenience 
we include a proof of the existence of nontrivial critical point of $I_\lambda$. 

We first show that $I_\lambda$ has a mountain pass geometry:

\begin{proposition}\label{proposition:3.1}
	Suppose {\rm (V1)--\rm (V2)} and {\rm (f1)--(f4)} hold. 
	Then for every $\lambda > 0$, there exist $\delta_\lambda >0$ and $\varphi_\lambda \in C^\infty_c(\RN)$ such that 
			\[
				\inf_{\| u \|_{H^1(\RN)} = \delta_\lambda } I_\lambda (u) > 0, 
				\quad \delta_\lambda < \| \varphi_\lambda \|_{H^1(\RN)}, \quad I_\lambda (\varphi_\lambda) < 0.
			\]
\end{proposition}

\begin{proof}
By (f2) and the definition of $g_{m_0}$, we may find $C_\lambda > 0$ so that 
	\[
		|g_{m_0}(s)| \leq \frac{V_0}{2\lambda} |s| + C_\lambda |s|^{p-1} 
		\quad \text{for each $s \in \R$}. 
	\]
Hence, it holds that 
	\[
		I_\lambda (u) \geq \int_{\RN} \frac{1}{4} | \nabla u|^2 + \frac{V(x)}{4} u^2 \, \rd x 
		- \frac{\lambda C_\lambda}{p} \| u \|_{L^p(\RN)}^p 
		\geq \frac{1}{4} \min \left\{ 1 , V_0 \right\} \| u \|_{H^1(\RN)}^2 - \wt{C}_\lambda \| u \|_{H^1(\RN)}^p
	\]
and the existence of $\delta_\lambda$ is easily deduced. 

Next, let $\varphi \in C^\infty_c(\RN)$ satisfy 
$\varphi \equiv 1$ on $B_1(0)$, $ 0 \leq \varphi \leq 1$ in $\RN$ and 
$\supp \varphi \subset B_2(0)$. By \eqref{eq:2.7}, for each $t \geq 0$, one has
	\[
		I_\lambda (t \varphi) \leq \frac{t^2}{2} \int_{\RN} |\nabla \varphi|^2 + V(x) \varphi^2 \, \rd x 
		- \wt{c} \lambda t^{p} \| \varphi \|_{L^{p}(\RN)}^{p}.
	\]
Thus, for sufficiently large $t$, we obtain 
$\delta_\lambda < \| t \varphi \|_{H^1(\RN)}$ and 
$I_\lambda (t \varphi) < 0$. 
\end{proof}

According to \eqref{eq:3.3} and \cref{proposition:3.1}, $J_\lambda$ also has a mountain pass geometry. 
For $I_\lambda$ and $J_\lambda$, we define the mountain pass values by 
	\begin{equation}\label{eq:3.4}
		\begin{aligned}
			c_\lambda &:= \inf_{ \gamma \in \Gamma_{I_\lambda}} \max_{0 \leq t \leq 1} I_\lambda (\gamma(t)) > 0, 
			& \Gamma_{I_\lambda} &:= \Set{ \gamma \in C \big( [0,1] , H^1(\RN) \big) | \gamma (0) = 0, \ I_\lambda (\gamma (1)) < 0 },
			\\
			d_\lambda &:= \inf_{ \gamma \in \Gamma_{J_\lambda}} \max_{0 \leq t \leq 1} J_\lambda (\gamma(t)), 
			& \Gamma_{J_\lambda} &:= \Set{ \gamma \in C \big( [0,1] , H^1(\RN) \big) | \gamma (0) = 0, \ J_\lambda (\gamma (1)) < 0 }.
		\end{aligned}
	\end{equation}
From \eqref{eq:3.3}, it follows that for each $\lambda > 0$, 
	\begin{equation}\label{eq:3.5}
		d_\lambda = \lambda^{2/(p-2)} c_\lambda. 
	\end{equation}
It is well-known that at the mountain pass level, there exists a Palais--Smale--Cerami sequence $(u_{\lambda,n})_n$ 
(see, for example, \cite{Ce78,Ek90,Ra12,St11}), that is, 
	\begin{equation}\label{eq:3.6}
		I_{\lambda} (u_{\lambda,n}) \to c_\lambda, \quad \left( 1 + \| u_{\lambda,n} \|_{H^1(\RN)} \right) 
		\left\| I_\lambda'(u_{\lambda,n}) \right\|_{(H^1(\RN))^\ast} \to 0.
	\end{equation}

	\begin{proposition}\label{proposition:3.2}
		Let {\rm (V1)--(V2)} and {\rm (f1)--(f4)} hold. 
		Then for each $\lambda > 0$, 
		the Palais--Smale--Cerami sequence $(u_{\lambda,n})_n$ in \eqref{eq:3.6} is bounded in $H^1(\RN)$. 
	\end{proposition}

\begin{proof}
We first claim that 
	\begin{equation}\label{eq:3.7}
		\text{for every $\tau > 0$, $ \left( \| u_{\lambda,n} \|_{L^p( [ u_{\lambda,n} \geq \tau ] )}  \right)_n$ is bounded},
	\end{equation}
where $[ u \geq c ] := \Set{ x \in \RN | u(x) \geq c }$. 
Let $\tau > 0$ be given. 
By \eqref{eq:3.6}, one has 
	\[
		c_\lambda + o(1) = I_{\lambda} (u_{\lambda,n}) - \frac{1}{2} I_{\lambda}'(u_{\lambda,n}) u_{\lambda,n} 
		= \lambda \int_{\RN} \frac{1}{2} g_{m_0} (u_{\lambda,n}) u_{\lambda,n} - G_{m_0} (u_{\lambda,n}) \, \rd x. 
	\]
Recalling \eqref{alignat:2.5}, \eqref{alignat:2.6} and noting $g_{m_0} \equiv 0$ in $(-\infty,0]$, we obtain 
	\[
		\lambda c_\tau \int_{ \left[ u_{\lambda,n} \geq \tau \right] } u_{n,\lambda}^{p} \, \rd x 
		\leq 
		\lambda \int_{\RN} \frac{1}{2} g_{m_0} (u_{\lambda,n}) u_{\lambda,n} - G_{m_0} (u_{\lambda,n}) \, \rd x
		= c_\lambda + o(1), 
	\]
which yields \eqref{eq:3.7}.

	Next by (f2) and the definition of $g_{m_0}$, there exist $\tau > 0$ and $C_{\lambda,\tau} > 0$ such that 
	\[
		|g_{m_0}(s)| \leq \frac{V_0}{2\lambda} |s| + C_{\lambda,\tau} \chi_{ [ \tau , \infty) } (s) s_+^{p-1} \quad 
		\text{for each $s \in \R$},
	\]
where $\chi_A(s)$ stands for the characteristic function of $A$.  
The above inequality implies that 
	\[
		\left| G_{m_0} (s) \right| 
		\leq 
		\frac{V_0}{4\lambda} s^2 + \frac{C_{\lambda,\tau}}{p} \chi_{ [\tau, \infty )  } (s) s_+^p 
		\quad \text{for all $s \in \R$}.
	\]
Therefore, one finds that
	\[
		c_\lambda + o(1) = I_\lambda(u_{\lambda,n}) 
		\geq \int_{\RN} \frac{1}{2} |\nabla u_{\lambda,n} |^2 + \frac{1}{4} V(x) u_{\lambda,n}^2 \, \rd x 
		- \frac{\lambda C_{\lambda, \tau}}{p} \| u_{\lambda,n} \|_{L^p( [ u_{\lambda,n} \geq \tau  ] )}^p.
	\]
Then by \eqref{eq:3.7}, we obtain the desired result. 
\end{proof}

To show that $(u_{\lambda,n})_n$ contains a strongly convergent subsequence in $H^1(\RN)$, 
we need some known facts. 

\begin{lemma}\label{lemma:3.3}
		Assume {\rm (V1)}, {\rm (V2)} and {\rm (f1)--(f4)}. 
		Let $\lambda > 0$ and $(u_n)_n \subset H^1(\RN)$ be a bounded Palais--Smale sequence of $I_\lambda$. 
		Then up to subsequences, 
		there exist $u_0 \in H^1(\RN)$, $k \geq 0$, $\omega_1,\dots,\omega_k \in H^1(\RN)$ and 
		sequences $(y_{1,n})_n,\dots, (y_{k,n})_n \subset \RN $ if $k \geq 1$ such that 
		\begin{enumerate}
		\item[{\rm (i)}]
		for each $i$, $|y_{i,n}| \to \infty$ and for each $i \neq j$, $|y_{i,n} - y_{j,n}| \to \infty$;
		\item[{\rm (ii)}]
		$u_n \rightharpoonup u_0$ weakly in $H^1(\RN)$, 
		$u_n (x+y_{i,n}) \rightharpoonup \omega_i \not \equiv 0$ weakly in $H^1(\RN)$, 
		$u_0$ is a solution of \eqref{eq:3.1} and $\omega_i$ satisfies 
			\begin{equation}\label{eq:3.8}
				-\Delta \omega_i + V_\infty \omega_i = \lambda g_{m_0} ( \omega_i ) \quad \text{in} \ \RN;
			\end{equation}
		\item[{\rm (iii)}]
		$\displaystyle \left\| u_n - u_0 - \sum_{i=1}^k \omega_i (\cdot - y_{i,n}) \right\|_{H^1(\RN)} \to 0$ 
		and $\displaystyle c_\lambda = I_\lambda (u_0) + \sum_{i=1}^k I_{\lambda,\infty} (\omega_i)$, where 
			\begin{equation}\label{eq:3.9}
				I_{\lambda,\infty} (u) := \frac{1}{2} \int_{\RN} |\nabla u|^2 + V_\infty u^2 \,\rd x
				- \lambda \intN G_{m_0} (u) \, \rd x.
			\end{equation}
		In particular, when $k=0$, $\| u_n - u_0 \|_{H^1(\RN)} \to 0$. 
		\end{enumerate}
	\end{lemma}

Since \cref{lemma:3.3} is well-known (we may argue as in \cite[Chapter 8]{Wi96}), 
we omit the proof.

	Next, we remark that $I_{\lambda,\infty}$ also admits a mountain pass geometry 
and let us denote by $c_{\lambda,\infty}$ its mountain pass value. 
In addition, it is known that \eqref{eq:3.8} admits a positive least energy solution $w_\lambda$, that is, 
$w_\lambda>0$ in $\RN$ and 
	\begin{equation}\label{eq:3.10}
		\inf \Set{ I_{\lambda,\infty} (u) | u \in H^1(\RN) \setminus \{0\}, \ I_{\lambda,\infty}'(u) = 0 } 
		= c_{\lambda,\infty} = I_{\lambda,\infty} (w_\lambda) > 0.
	\end{equation}
For the existence, see \cite{BGK83,BL83} and for \eqref{eq:3.10}, we refer to \cite{JT03a,JT03b}. 
Furthermore, according to \cite[Lemma 2.1]{JT03a} and \cite[Section 3]{JT03b}, 
there exist $\gamma_{\lambda,\infty} \in C( [0,1] , H^1(\RN) )$ such that 
	\begin{equation}\label{eq:3.11}
		\begin{aligned}
			&
			w_\lambda \in \gamma_{\lambda,\infty} \left( \left[ 0 , 1 \right] \right), 
			\quad 
			\gamma_{\lambda,\infty} (t) (x) > 0 \quad \text{for each $t \in (0,1]$ and $x \in \RN$},
			\\
			&I_{\lambda,\infty} \left( \gamma_{\lambda,\infty} (1) \right) < 0, \quad 
			c_{\lambda,\infty} = \max_{0 \leq t \leq 1} I_{\lambda,\infty} \left( \gamma_{\lambda,\infty} (t) \right) 
			> I_{\lambda,\infty} \left( \gamma_{\lambda,\infty} (s) \right) 
			\quad \text{if $\gamma_{\lambda,\infty} (s) \neq w_{\lambda}$}.
		\end{aligned}
	\end{equation}

Under these preparations, we prove the following.

\begin{proposition}\label{proposition:3.4}
		Suppose {\rm (V1)--(V3)} and {\rm (f1)--(f4)}. 
		Then, up to subsequences, there exists $u_\lambda \in H^1(\RN)$ 
		such that $\| u_{\lambda,n} - u_\lambda \|_{H^1(\RN)} \to 0$ as $n \to \infty$. 
		Moreover, $u_{\lambda}$ is a positive solution of \eqref{eq:3.1}. 
	\end{proposition}

\begin{proof}
Without loss of generality, we may assume $V(x) \not \equiv V_{\infty}$. 
Otherwise, the existence of the mountain pass solutions of \eqref{eq:3.1} follows from \cite{JT03a,JT03b}. 
By (V1)--(V3), we have 
	\[
		I_{\lambda} (u) \leq I_{\lambda,\infty} (u) \quad \text{for all $u \in H^1(\RN)$}. 
	\]
In particular, $I_{\lambda} (\gamma_{\lambda,\infty} (1)) \leq I_{\lambda,\infty} (\gamma_{\lambda,\infty} (1)) < 0$ 
thanks to \eqref{eq:3.11}. 
Thus, it holds that $\gamma_{\lambda,\infty} \in \Gamma_{I_\lambda}$. 
Let $t_\lambda \in (0,1)$ satisfy 
$ I_{\lambda} ( \gamma_{\lambda,\infty} (t_\lambda) ) = \max_{0 \leq t \leq 1} I_\lambda ( \gamma_{\lambda,\infty} (t) )$. 
Then (V3) with $V \not \equiv V_\infty$ and \eqref{eq:3.11} give 
	\begin{equation}\label{eq:3.12}
		c_{\lambda,\infty} = I_{\lambda,\infty} (w_\lambda) 
		\geq I_{\lambda,\infty} \left( \gamma_{\lambda,\infty}(t_\lambda) \right) 
		> I_{\lambda} ( \gamma_{\lambda,\infty} (t_\lambda) ) \geq c_{\lambda}.
	\end{equation}

We apply \cref{lemma:3.3} to $(u_{\lambda,n})$ and let $k \geq 0$, $u_0$, $\omega_i$ and $(y_{i,n})$ be as in \cref{lemma:3.3}. 
We prove $k = 0$ and assume by contradiction that $k \geq 1$. 
Since $\omega_i \not \equiv 0$ and $I_{\lambda,\infty} '(\omega_i) = 0$, \eqref{eq:3.10} yields 
$I_{\lambda,\infty} (\omega_i) \geq c_{\lambda, \infty}$. Therefore, one has 
	\[
		c_\lambda = I_{\lambda} (u_0) + \sum_{i=1}^k I_{\lambda,\infty} (\omega_i) 
		\geq I_{\lambda} (u_0) + k c_{\lambda,\infty}. 
	\]
On the other hand, since $I_{\lambda}'(u_0) = 0$ and 
$ 0 \leq \frac{1}{2} g_{m_0} (s) s - G_{m_0} (s)$ for each $s \in \R$ in view of \eqref{alignat:2.6}, we see that
	\[
		I_{\lambda} (u_0) = I_{\lambda} (u_0) - \frac{1}{2} I_{\lambda}'(u_0) u_0
		= \lambda \int_{\RN} \frac{1}{2} g_{m_0} (u_0) u_0 - G_{m_0} (u_0) \, \rd x \geq 0,
	\]
which leads to a contradiction by \eqref{eq:3.12} and $c_{\lambda,\infty} > 0$ as follows:
	\[
		c_{\lambda,\infty} > c_{\lambda} \geq k c_{\lambda,\infty} \geq c_{\lambda,\infty}.
	\]
Hence, $k = 0$ and there exists $u_{\lambda} \in H^1(\RN)$ such that 
$\| u_{\lambda,n} - u_{\lambda} \|_{H^1(\RN)} \to 0$. 

By $g_{m_0} (s) \equiv 0$ on $(-\infty,0]$ and $I_{\lambda} (u_\lambda) = c_\lambda > 0$, 
the weak maximum principle yields that $u_\lambda \geq 0$ in $\RN$ and $u_\lambda \not \equiv 0$. 
From the growth condition on $g_{m_0}$, due to Moser's iteration and elliptic regularity, 
we can verify that $u_\lambda \in L^\infty(\RN)$, 
and hence $u_\lambda \in W^{2,p}_{\rm loc} (\RN)$ for each $p \in (1,\infty)$. 
Noting that 
	\[
		-\Delta u_\lambda + \left( V(x) - \lambda \frac{g_{m_0} (u_\lambda) }{u_\lambda} \right) u_{\lambda} = 0,
	\]
we infer from the weak Harnack inequality (\cite[Theorem 8.20]{GT01}) that $u_\lambda > 0$ in $\RN$. 
Thus, we complete the proof. 
\end{proof}

Our next aim is to show $\| u_{\lambda} \|_{L^\infty(\RN)} \to 0$ as $\lambda \to \infty$. 
To this end, we first estimate $(c_\lambda)_{\lambda}$.

\begin{lemma}\label{lemma:3.5}
		Let {\rm (V1)--(V2)} and {\rm (f1)--(f4)} hold.  
		Then there exists $M>0$ such that for all $\lambda > 0$, 
		$c_\lambda \leq M \lambda^{-2/(p-2)}$ holds. 
\end{lemma}

\begin{proof}
From \eqref{eq:3.5}, it suffices to show $d_\lambda \leq M$ for any $\lambda > 0$. 
Recalling \eqref{eq:2.7}, we set 
	\[
		K_0(v) := \frac{1}{2} \int_{\RN} |\nabla v|^2 + V(x) v^2 \,\rd x - \wt{c} \int_{\R^N} v_+^p \, \rd x \in C^1 \big( H^1(\RN), \R \big).
	\]
Then $J_\lambda (v) \leq K_0(v)$ holds for every $v \in H^1(\RN)$. 
Moreover, it is easy to show that $K_0$ has a mountain pass geometry and hence we can define 
	\[
		\begin{aligned}
			\ov{d}_0 &:= \inf_{ \gamma \in \ov{\Gamma}_0 } \max_{0 \leq t \leq 1} 
			K_0 ( \gamma (t) ) \in (0,\infty), 
			\\
			\ov{\Gamma}_0 &:= \Set{ \gamma \in C \big( [0,1] , H^1(\RN) \big) | \gamma (0) = 0 , \ K_0(\gamma (1)) < 0  }.
		\end{aligned}
	\]
Since $\ov{\Gamma}_0 \subset \Gamma_{J_\lambda}$ due to $J_\lambda (v) \leq K_0(v)$, 
we obtain $d_\lambda \leq \ov{d}_0$ and this completes the proof. 
\end{proof}

	Next, we show that any family $(u_\lambda)_{\lambda > 0}$ of solutions of \eqref{eq:3.1} 
with $I_{\lambda} (u_\lambda) \leq M \lambda^{-2/(p-2)}$ satisfies 
$\| u_\lambda \|_{L^\infty(\RN)} \to 0$ as $\lambda \to \infty$. 

\begin{proposition}\label{proposition:3.6}
		Suppose {\rm (V1)--(V2)} and {\rm (f1)--(f4)}. 
		Let $(u_\lambda)_{\lambda > 0}$ be a family of solutions of \eqref{eq:3.1} 
		satisfying $I_{\lambda} (u_\lambda) \leq M\lambda^{-2/(p-2)}$. 
		Then $\| u_\lambda \|_{L^\infty (\RN)} \to 0$ as $\lambda \to \infty$. 
	\end{proposition}

\begin{proof}
Let $\tau \in (0,s_0)$ be an arbitrary number. 
By the assumption on $(u_\lambda)$, we have
	\[
		\lambda \int_{ \RN} \frac{1}{2} g_{m_0} (u_\lambda) u_\lambda - G_{m_0} (u_\lambda) \, \rd x 
		= I_{\lambda} (u_\lambda) - \frac{1}{2} I_{\lambda}'(u_\lambda) u_\lambda 
		= I_{\lambda} (u_\lambda)
		\leq M \lambda^{-2/(p-2)} \quad \text{for any $\lambda > 0$}.
	\]
Hence, \eqref{alignat:2.5} and $0 \leq \frac{1}{2} g_{m_0} (s) s - G_{m_0} (s)$ for all $s \in \R$ give
	\[
		c_\tau \int_{ [u_\lambda \geq \tau] } u_\lambda^p \, \rd x 
		\leq \int_{ \RN} \frac{1}{2} g_{m_0} (u_\lambda) u_\lambda - G_{m_0} (u_\lambda) \, \rd x 
		\leq M \lambda^{ -p/(p-2) }. 
	\]
On the other hand, testing $(u_\lambda - \tau)_+ \in H^1(\RN)$ to \eqref{eq:3.1} and noting that
$g_{m_0} (s) (s-\tau)_+ \leq C_\tau s^p \chi_{ [\tau,\infty) } (s)  $ due to \eqref{eq:2.4} and that 
	\[
		\nabla u_\lambda \cdot \nabla (u_\lambda -\tau)_+ + V(x) u_\lambda (u_\lambda - \tau)_+ 
		\geq \left| \nabla \left( u_\lambda -\tau \right)_+ \right|^2 
		+ V(x) \left( u_\lambda - \tau \right)_+^2 \quad \text{a.e. in $\RN$},
	\]
we obtain 
	\begin{equation}\label{eq:3.13}
		\begin{aligned}
			\int_{ \RN} \left| \nabla \left( u_\lambda - \tau \right)_+ \right|^2 + V(x) \left( u_\lambda - \tau \right)_+^2 \rd x
			&\leq 
			\lambda \int_{ \RN} g_{m_0} (u_\lambda) (u_\lambda - \tau)_+ \, \rd x 
			\\
			&\leq C_\tau \lambda \int_{[u_\lambda \geq \tau]} u_\lambda^p\, \rd x 
			\leq M \wt{C}_\tau \lambda^{-2/(p-2)}.
		\end{aligned}
	\end{equation}
When $N=1$, since $\| u \|_{L^\infty(\R)} \leq C \| u \|_{H^1(\R)}$, 
\eqref{eq:3.13} yields $\| (u_\lambda - \tau)_+ \|_{L^\infty(\R)} \to 0$ as $\lambda \to \infty$. 
Since $\tau \in (0,s_0)$ is arbitrary, we conclude that $\| u_\lambda \|_{L^\infty(\R)} \to 0$ as $\lambda \to \infty$. 

Next, we consider the case $N \geq 2$. 
By $u_+ = ( |u| + u )/2$, Kato's inequality \cite[Lemma A]{Ka} and $V(x) \geq 0$, $(u_\lambda - \tau)_+ \in H^1(\RN)$ satisfies 
the following inequality in a weak sense:
	\[
		\begin{aligned}
			-\Delta \left( (u_\lambda - \tau)_+ \right) 
			&\leq 
			\left\{ -\Delta (u_\lambda - \tau) \right\} 
			\sgn^+ \left( \left( u_\lambda - \tau \right)_+ \right)
			\\
			&\leq 
			\lambda \frac{g_{m_0} (u_\lambda) }{u_\lambda} \chi_{ [u_\lambda \geq \tau] } 
			\left\{ (u_\lambda - \tau)_+ + \tau \right\} 
			\quad \text{in} \ \RN,
		\end{aligned}
	\]
where $\sgn^+(a) = 1$ if $a>0$, $\sgn^+(0) = 1/2$ and $\sgn^+(a) = 0$ if $a < 0$. 
Writing 
	\[
		\phi_\lambda (x) := (u_\lambda(x) - \tau )_+, \quad 		
		h_\lambda (x) := \lambda \frac{g_{m_0} (u_\lambda (x)) }{u_\lambda(x)} \chi_{ [u_\lambda \geq \tau] } (x),
	\]
we get 
	\begin{equation}\label{eq:3.14}
		-\Delta \phi_\lambda \leq h_\lambda \phi_\lambda + \tau h_\lambda \quad \text{in} \ \RN.
	\end{equation}

By $|g_{m_0}(s)| \leq C_\tau s^{p-1}$ for $s \geq \tau$, one has
	\[
		0 \leq h_\lambda (x) \leq C_\tau \lambda u_\lambda^{p-2} (x) \chi_{ [u_\lambda \geq \tau] } (x).
	\]
Moreover, from $p-2 < 4/(N-2)$, we may find $q_1 < q_2$ so that 
	\[
		\frac{p-2}{2^\ast} < \frac{1}{q_2} < \frac{1}{q_1} < \min \left\{ \frac{2}{N} , \frac{p-2}{2} \right\}.
	\]
Since $\| (u_\lambda - \tau/2)_+ \|_{H^1(\RN)} \leq M \wt{C}_{\tau/2} \lambda^{-1/(p-2)}$ by \eqref{eq:3.13}, 
it follows from 
$s^{(p-2)q_1} \chi_{[s \geq \tau]}(s) \leq C_\tau (s- \tau/2)_+^{(p-2)q_2}$ and $2 < q_2(p-2) < 2^\ast$ that 
	\[
		\begin{aligned}
			\left\| \lambda u_\lambda^{p-2} \chi_{ [u_\lambda \geq \tau] } \right\|_{L^{q_1} (\RN) }^{q_1} 
			&=
			\lambda^{q_1} \int_{ \RN} u_\lambda^{(p-2)q_1} \chi_{ [u_\lambda \geq \tau] } \, \rd x
			\\
			&\leq
			\lambda^{q_1} \int_{ \RN} C_\tau \left( u_\lambda - \frac{\tau}{2} \right)^{(p-2)q_2}_+ \, \rd x
			\\
			&\leq C_\tau \lambda^{q_1} C_0 \left\| \left( u_\lambda - \frac{\tau}{2} \right)_+ \right\|_{H^1(\RN)}^{(p-2)q_2} 
			\leq \wt{C}_\tau \lambda^{ q_1 - q_2 }.
		\end{aligned}
	\]
Hence, $\| h_\lambda \|_{L^{q_1} (\RN) } \to 0$ as $\lambda \to \infty$. 
By \eqref{eq:3.14}, $q_1 > N/2$ and Moser's iteration (\cite[Theorem 8.15]{GT01} or \cite[Theorem 4.1]{HL01}), 
there exists $C>0$, which is independent of $\lambda$, such that 
for any $z \in \RN$, 
	\[
		\| \phi_\lambda \|_{L^\infty(B_1(z))} \leq 
		C \left\{ \left\| \phi_\lambda \right\|_{L^2(B_2(z))} + \left\| \tau h_\lambda \right\|_{L^{q_1} (B_2(z)) } \right\}.
	\]
Since $\| \phi_\lambda \|_{L^2(\RN)} \leq M \wt{C}_{\tau/2} \lambda^{-1/(p-2)}$, we obtain 
	\[
		\limsup_{\lambda \to \infty} \| \phi_\lambda \|_{L^\infty(\RN)} 
		= \limsup_{\lambda \to \infty} \sup_{z \in \RN} \| \phi_\lambda \|_{L^\infty(B_1(z))} = 0.
	\]
Recalling that $\tau \in (0,s_0)$ is arbitrary, we have $\| u_\lambda \|_{L^\infty(\RN)} \to 0$ as $\lambda \to \infty$. 
	\end{proof}

Now we prove \cref{theorem:1.1} (i).
	
\begin{proof}[Proof of \cref{theorem:1.1} (i)]
Let $(u_\lambda)_{\lambda > 0}$ be a family of positive solutions of \eqref{eq:3.1} obtained in \cref{proposition:3.4}. 
By \cref{lemma:3.5} and \cref{proposition:3.6}, 
we have $\| u_\lambda \|_{L^\infty(\RN)} \to 0$ as $\lambda \to \infty$. 
Hence, we can select $\lambda_0 > 0$ such that 
$\| u_\lambda \|_{L^\infty(\RN)} \leq s_0$ if $\lambda \geq \lambda_0$. 
By \eqref{eq:2.4}, one finds that $g_{m_0} (u_\lambda) = f(u_\lambda)$ and 
$u_\lambda$ is a positive solution of \eqref{Plam}. 
Therefore, under (i) \cref{theorem:1.1} holds. 
\end{proof}

\subsection{Proof of \cref{theorem:1.1} under (ii)}
\label{section:3.2}

	Here we give a sketch of proof of \cref{theorem:1.1} under (ii). 
Assume that (V1), (V4) and (f1)--(f4). 
In this case, we define $H_V$ by 
	\[
		\begin{aligned}
			H_V &:= \Set{ u \in H^1(\RN) | \int_{ \RN} V(x) u^2 \, \rd x < \infty }, 
			\\
			\la u , v \ra_{H_V} 
			&:= \int_{ \RN} \nabla u \cdot \nabla v + V(x) u v \, \rd x, \quad 
			\| u \|_{H_V} := \sqrt{ \la u, u \ra_{H_V} }.
		\end{aligned}
	\]
Then $(H_V, \la \cdot, \cdot\ra_{H_V} )$ is a Hilbert space. 
Moreover, for $I_\lambda$ in \eqref{eq:3.2}, we may verify $I_\lambda \in C^1(H_V,\R)$ and 
critical points of $I_\lambda$ in $H_V$ are solutions of \eqref{eq:3.1} in $H_V$.

	As in \cref{proposition:3.1}, it is not difficult to prove that $I_\lambda$ admits a mountain pass geometry in $H_V$. 
Therefore, we may define the mountain pass value of $I_\lambda$ in $H_V$ by: 
	\[
		\begin{aligned}
			c_{V,\lambda} &:= \inf_{\gamma \in \Gamma_{V,\lambda}} 
			\max_{0 \leq t \leq 1} I_\lambda \left( \gamma (t) \right) \in (0,\infty),
			\\
			\Gamma_{V,\lambda} &:= 
			\Set{ \gamma \in C \left( [0,1] , H_V \right) | 
			\gamma (0) = 0, \ I_{\lambda} (\gamma (1)) < 0
			 }.
		\end{aligned}
	\]
In addition, we may find a Palais--Smale--Cerami sequence of $(u_{\lambda,n})_n \subset H_V$ at level $c_{V,\lambda}$, that is, 
	\[
		I_{\lambda} (u_{\lambda,n}) \to c_{V,\lambda}, \quad 
		\left( 1 + \| u_{\lambda,n} \|_{H_V} \right) \left\| I_\lambda'(u_{\lambda,n}) \right\|_{H_V^\ast} \to 0
		\quad \text{as $n \to \infty$}. 
	\]
Then the argument of \cref{proposition:3.2} is still valid in this case and 
we find that $(u_{\lambda,n})_n$ is bounded in $H_V$. 
Without loss of generality, we may suppose that 
$u_{\lambda,n} \rightharpoonup u_\lambda$ weakly in $H_V$ 
and $u_{\lambda,n} (x) \to u_\lambda (x)$ a.e. $x \in \RN$.

	Next, we recall the compactness embedding of $H_V$ into $L^q(\RN)$ with $2 \leq q < 2^\ast$, 
which was proved in \cite[Lemma 3.1]{BaWaWi05}. 
	\begin{lemma}[]\label{lemma:3.7}
		Under {\rm (V1)} and {\rm (V4)}, the embedding $H_V \subset L^q(\RN)$ is compact 
		for all $q \in [2,2^\ast)$. 
	\end{lemma}

	Exploiting \cref{lemma:3.7}, (f1), (f2) and \cref{lemma:2.1}, we may show that 
$u_{\lambda,n} \to u_\lambda$ strongly in $L^q(\RN)$ for each $q \in [2,2^\ast)$ and 
$g_{m_0} (u_{\lambda,n}) \to g_{m_0} (u_\lambda)$ strongly in $H_V^\ast$. 
By $I_\lambda'(u_{\lambda,n}) \to 0$ strongly in $H_V^\ast$ and 
the isomorphism between $H_V$ and $H_V^\ast$, we observe that 
$u_{\lambda,n} \to u_\lambda$ strongly in $H_V$.

	Finally, we show $\| u_\lambda \|_{L^\infty(\RN)} \to 0$ as $\lambda \to \infty$. 
For \cref{lemma:3.5}, considering $K_0$ on $H_V$ leads to $c_{V,\lambda} \leq C \lambda^{-2/(p-2)}$ 
for each $\lambda > 0$. 
Furthermore, the argument in \cref{proposition:3.6} also works by replacing the $H^1(\RN)$-norm 
by the $H_V$-norm. 
Thus, $\| u_\lambda \|_{L^\infty(\RN)} \to 0$ as $\lambda \to \infty$ holds and 
\cref{theorem:1.1} follows under (ii).

\subsection{The existence of positive solution of \eqref{eq:1.1} under (f1)--(f4)}
\label{section:3.3}

	When $\Omega \subset \RN$ is a bounded domain, we suppose  $V (x) \in C(\ov{\Omega})$ and $V(x) \geq 0$, 
and consider the functional 
	\[
		\wt{I}_\lambda (u) := \int_{\Omega} \frac{1}{2} |\nabla u|^2 + V(x) u^2 \, \rd x - \lambda \int_{\Omega} G_{m_0} (u) \, \rd x 
		\in C^1( H^1_0(\Omega) , \R).
	\]
We equip $H_0^1(\Omega)$ with the norm $\| \nabla \cdot \|_{L^2(\Omega)}$. 
Then it is easily seen that $\wt{I}_\lambda$ has a mountain pass geometry on $H^1_0(\Omega)$ and 
the mountain pass value is well-defined:
	\[
		\begin{aligned}
			c_{\Omega,\lambda} &:= \inf_{\gamma \in \Gamma_{\Omega,\lambda}} 
			\max_{0 \leq t \leq 1} \wt{I}_\lambda \left( \gamma (t) \right) \in (0,\infty),
			\\
			\Gamma_{\Omega,\lambda} &:= 
			\Set{ \gamma \in C \left( [0,1] , H^1_0(\Omega) \right) | 
			\gamma (0) = 0, \ \wt{I}_{\lambda} (\gamma (1)) < 0
			 }.
		\end{aligned}
	\]
Thus, there exists a Palais--Smale--Cerami sequence $(u_{\lambda,n})_n$ at level $c_{\Omega,\lambda}$ and 
as in the above we may show $\| u_{\lambda,n} - u_\lambda \|_{H^1_0(\Omega)} \to 0$ as $n \to \infty$ 
after taking a subsequence. 
\cref{lemma:3.5} remains true in this case and we obtain $c_{\Omega, \lambda} \leq C \lambda^{-2/(p-2)}$. 
Finally, for the assertion $\| u_\lambda \|_{L^\infty(\Omega)} \to 0$ as $\lambda \to \infty$, we remark that 
$\phi_\lambda (x) := (u_\lambda - \tau)_+$ with $0< \tau < s_0$ satisfies 
	\[
		-\Delta \phi_\lambda \leq h_\lambda \phi_\lambda + \tau h_\lambda \quad \text{in} \ \Omega, 
		\quad 
		h_\lambda (x) := \lambda \frac{g_{m_0} (u_\lambda (x)) }{u_\lambda (x)} \chi_{ [u_\lambda \geq \tau] } (x).
	\]
As in the previous case, we can show 
$\| h_{\lambda} \|_{L^{q_1} (\RN) } \to 0$ as $\lambda \to \infty$ for some $q_1 > N/2$. 
Since $\phi_\lambda \in H^1_0(\Omega)$, we may apply Moser's iteration without cut-off function in $x$ and get 
	\[
		\| \phi_\lambda \|_{L^\infty(\Omega)} \leq C \left\{ \| \phi_\lambda \|_{L^2(\Omega)} + \| \tau h_\lambda \|_{L^{q_1} (\Omega) } \right\},
	\]
where $C>0$ is independent of $\lambda$. 
Therefore, we have $\| u_\lambda \|_{L^\infty(\Omega)} \to 0$ as $\lambda \to \infty$ 
and $u_\lambda$ becomes a solution of \eqref{Plam} when $\lambda$ is sufficiently large.

\section{Asymptotic behavior of positiver solutions as $\lambda \to \infty$: proof of \cref{theorem:1.3}}
\label{section:4}

In this section, we aim to prove \cref{theorem:1.3}. 
As in \cref{section:3}, we mainly consider the case (V1)--(V3), (f1) and (f5) hold 
since other cases can be shown similarly. 
Recalling that (f1) and (f5) imply (f2)--(f4), 
we notice that the mountain pass values $c_\lambda$ and $d_\lambda$ in \eqref{eq:3.4} are 
well-defined and they are critical values of $I_\lambda$ and $J_\lambda$ respectively. 
Let $u_\lambda$ be a critical point of $I_\lambda$ corresponding to $c_\lambda$ 
and $\omega_0$ be the unique positive radial solution of \eqref{eq:1.3}.  
We choose $T_0>0$ so that 
	\[
		J_{\infty, \infty} ( T_0 \omega_0 ) < 0, \quad \hbox{where} \ 
		J_{\infty, \infty} ( u ) := \frac{1}{2} \intN |\nabla u|^2 + V_\infty u^2 \,\rd x - \frac{1}{p} \intN u_+^p \, \rd x.
	\]
We also set 
	\[
		\gamma_\infty(t) := T_0 t \omega_0 \in C \big( [0,1] , H^1(\RN) \big), \quad 
		J_\infty (u) := \frac{1}{2} \intN |\nabla u|^2 + V(x) u^2 \,\rd x - \frac{1}{p} \intN u_+^p \, \rd x.
	\]

\begin{lemma}\label{lemma:4.1}
		Suppose {\rm (V1)--(V3)}, {\rm (f1)} and {\rm (f5)}, and set $v_\lambda (x) := \lambda^{1/(p-2)} u_\lambda (x)$. 
		Then it holds
			\[
				0 < \liminf_{\lambda \to \infty} \lambda^{2/(p-2)} c_\lambda \leq 
				\limsup_{\lambda \to \infty} \lambda^{2/(p-2)} c_\lambda 
				\leq \max_{ 0 \leq t \leq 1} J_\infty ( \gamma_\infty (t) )
			\]
		and $(v_\lambda)_{\lambda \geq 1}$ is bounded in $H^1(\RN)$. 
	\end{lemma}

	\begin{proof}
To show $\dis 0 < \liminf_{\lambda \to \infty} \lambda^{2/(p-2)} c_\lambda$, 
let $C_2>0 $ be the constant in \eqref{alignat:2.9} and set 
	\[
		\ud{J} (v) := \frac{1}{2} \intN | \nabla v |^2 + V(x) v^2 \,\rd x - \frac{C_2}{p} \intN v_+^p \, \rd x.
	\]
Then we have 
	\begin{equation}\label{eq:4.1}
		\ud{J} (v) \leq J_\lambda (v) \quad \text{for each $v \in H^1(\RN)$ and $\lambda > 0$}.
	\end{equation}
It is easily seen that $\ud{J}$ has a mountain pass geometry and we set 
	\[
		\ud{d} := \inf_{\gamma \in \ud{\Gamma}} \max_{ 0 \leq t \leq 1} \ud{J}(\gamma(t)) \in (0,\infty), 
		\quad 
		\ud{\Gamma} := \Set{ \gamma \in C \big( [0,1] , H^1(\RN) \big) | \gamma (0) = 0, \ \ud{J} (\gamma (1)) < 0 }.
	\]
From \eqref{eq:4.1}, it follows that 
	\[
		\Gamma_{J_\lambda} \subset \ud{\Gamma}, \quad 
		0 < \ud{d} \leq d_\lambda \quad \text{for all $\lambda > 0$}.
	\]
Hence, \eqref{eq:3.5} yields $\ud{d} \leq c_\lambda \lambda^{2/(p-2)} $ and hence 
$\dis  0 < \liminf_{\lambda \to \infty} \lambda^{2/(p-2)} c_\lambda$ follows. 

Next, since $G_{m_0} (s) = F(s)$ on $[0,s_0]$, we notice that 
for sufficiently large $\lambda$ and all $ t \in [0,1]$,
		\[
			\begin{aligned}
				J_\lambda ( \gamma_\infty(t) ) 
				&= 
				\frac{T_0^2t^2}{2} \intN |\nabla \omega_0 |^2 + V(x) \omega_0^2 \,\rd x 
				-\lambda^{p/(p-2)} \intN G_{m_0} \left( \lambda^{-1/(p-2)} T_0 t \omega_0 (x) \right) \rd x 
				\\
				&= 
				\frac{T_0^2t^2}{2} \intN |\nabla \omega_0 |^2 + V(x) \omega_0^2 \,\rd x
				- \intN \frac{ F \left( \lambda^{-1/(p-2)  } T_0 t \omega_0 (x) \right) }
				{ \left( \lambda^{-1/(p-2)} T_0 t \omega_0 (x) \right)^p } \left( T_0 t \omega_0 (x) \right)^p \rd x.
			\end{aligned}
		\]
	Thus by (f5), one finds that
		\[
			J_\lambda(\gamma_\infty (t) ) \to J_\infty(\gamma_\infty (t)) 
			\ \hbox{as $\lambda \to \infty$} \ \text{uniformly with respect to $t \in [0,1]$}.
		\]
Due to $J_\infty(u) \leq J_{\infty,\infty} (u)$ and $J_{\infty,\infty} (T_0 \omega_0) < 0$, we deduce that
		\[
			J_\lambda (\gamma_\infty (1)) < 0 \quad 
			\text{for sufficiently large $\lambda$}, \quad 
			\lim_{\lambda \to \infty} \max_{ 0 \leq t \leq 1} J_\lambda (\gamma_\infty (t)) 
			= \max_{0 \leq t \leq 1} J_{\infty} \left( \gamma_\infty (t) \right).
		\]
Hence, it follows that $\gamma_\infty \in \Gamma_{J_\lambda}$ for sufficiently large $\lambda$ and 
	\[
		\limsup_{\lambda \to \infty} d_\lambda \leq \limsup_{\lambda \to \infty} \max_{0\leq t \leq 1} 
		J_\lambda ( \gamma_\infty (t) ) = \max_{ 0 \leq t \leq 1} J_\infty(\gamma_\infty (t)). 
	\]
Then from \eqref{eq:3.5}, 
$\dis \limsup_{\lambda \to \infty} \lambda^{2/(p-2)} c_\lambda \leq \max_{ 0 \leq t \leq 1} J_\infty(\gamma_\infty (t) )$ holds. 

For the boundedness of $(v_\lambda)$, let $\mu \in (2,p)$ be the constant in \eqref{alignat:2.8}. 
Then \eqref{alignat:2.8} gives 
	\[
		c_\lambda = I_\lambda (u_\lambda) = I_\lambda (u_\lambda) - \frac{1}{\mu} I_\lambda'(u_\lambda) u_\lambda 
		\geq \left( \frac{1}{2} - \frac{1}{\mu} \right) \int_{ \RN} |\nabla u_\lambda|^2 + V(x) u_\lambda^2 \, \rd x.
	\]
Multiplying $\lambda^{2/(p-2)}$ and noting that $( \lambda^{ 2/(p-2) }  c_\lambda)$ is bounded, 
we obtain the desired result. 
	\end{proof}

Now we are ready to prove \cref{theorem:1.3}. 
	
\begin{proof}[Proof of \cref{theorem:1.3}]
Let $(\lambda_n)_n \subset (0,\infty)$ be any sequence satisfying $\lambda_n \to \infty$ 
and set $v_n(x) := \lambda_n^{1/(p-2)} u_{\lambda_n} (x)$.
By \eqref{eq:3.3}, one finds that
	\begin{equation}\label{eq:4.2}
		J_{\lambda_n} (v_n) = \lambda_n^{2/(p-2)} I_{\lambda_n} \left( \lambda_n^{-1/(p-2)} v_n  \right) 
		= \lambda_n^{2/(p-2)} c_{\lambda_n}, \quad J_{\lambda_n}'(v_n) = 0.
	\end{equation}
Notice that in view of \eqref{alignat:2.9},
	\[
		- \Delta v_n + V(x) v_n = \lambda_n^{(p-1)/(p-2)} g_{m_0} \left( \lambda_n^{-1/(p-2)} v_n \right) 
		\leq C_2 v_n^{p-1} \quad \text{in} \ \RN.
	\]
Since $(v_n)_n$ is bounded in $H^1(\RN)$ thanks to \cref{lemma:4.1}, 
elliptic regularity theory implies that 
$(v_n)_n$ is bounded in $L^\infty(\RN)$, and hence in $W^{2,q}_{\rm loc} (\RN)$ for each $q \in (1,\infty)$. 

We first show the following claim:

\medskip

\noindent
\textbf{Claim:} 
\textsl{There exist $v_0 \in H^1(\RN)$, $k \geq 0$ and $(y_{1,n})_n,\dots,(y_{k,n})_n \subset \RN$ if $k \geq1$ such that 
	\begin{enumerate}
	\item[{\rm (i)}] 
	for each $i$, $|y_{i,n}| \to \infty$ and for each $i \neq j$, $|y_{i,n} - y_{j,n}| \to \infty$.
	\item[{\rm (ii)}]
	$v_n \rightharpoonup v_0 \geq 0$ and $v_n(x+y_{i,n}) \rightharpoonup \omega_0 \not \equiv 0$ 
	weakly in $H^1(\RN)$ where $\omega_0$ is the unique positive radial solution of \eqref{eq:1.3} and 
	$v_0$ satisfies 
		\[
			-\Delta v_0 + V(x) v_0 = v_0^{p-1} \quad \text{in} \ \RN.
		\]
	\item[{\rm (iii)}]	
	$\displaystyle \left\| v_n - v_0 - \sum_{i=1}^k \omega_0 ( \cdot - y_{i,n} ) \right\|_{H^1(\RN)} \to 0$ 
	and $\displaystyle \lim_{n \to \infty} J_{\lambda_n} (v_n) = J_\infty(v_0) + k J_{\infty,\infty} (\omega_0)$. 
	\end{enumerate}
}

\begin{proof}[Proof of the claim]
We recall that $v_n>0$ in $\RN$ and $(v_n)_n$ is bounded in $H^1(\RN) \cap L^\infty(\RN)$. 
Hence, for any $(z_n)_n \subset \RN$ with $\| v_n \|_{L^2(z_n+Q)} \to c > 0$ and $Q:= [0,1]^N$, 
$(v_n(\cdot + z_n))_n$ is bounded in $W^{2,q}_{\rm loc} (\RN)$ 
for each $q < \infty$. 
Let $v_n(x + z_n) \rightharpoonup \omega \geq 0$ weakly in $H^1(\RN)$ and $W^{2,q}_{\rm loc} (\RN)$. 
Then it holds that $ v_n(x+z_n) \to  \omega \not \equiv 0$ in $C_{\rm loc} (\RN)$ and hence (f5) yields 
	\[
		\begin{aligned}
			&\lambda_n^{ (p-1)/ (p-2) }  g_{m_0} \left( \lambda_n^{-1/(p-2)} v_n(x) \right) 
			= \frac{g_{m_0} \left( \lambda_n^{-1/(p-2)} v_n(x) \right)  }{\left( \lambda_n^{-1/(p-2)} v_n(x) \right) ^{p-1}} 
			v_n^{p-1} (x)
			\to \omega^{p-1} (x) \quad \text{in} \ C_{\rm loc} (\RN),
			\\
			& \lambda_n^{ p/(p-2) } G_{m_0} \left( \lambda_n^{-1/(p-2)} v_n(x) \right) 
			\to \frac{1}{p}\omega^{p} (x) \quad \text{in} \ C_{\rm loc} (\RN). 
		\end{aligned}
	\]
This implies that when $|z_n| \to \infty$, $\omega$ is a nontrivial, nonnegative solution of 
$-\Delta \omega + V_\infty \omega = \omega^{p-1}$ in $\RN$, and thus 
$\omega(x) = \omega_0 (x-z_0)$ for some $z_0 \in \RN$ thanks to the strong maximum principle and 
the uniqueness of positive solution up to translations (\cite{Kw89}). 
Replacing $z_n$ by $z_n + z_0$, we observe that 
$v_n(x+z_n) \rightharpoonup \omega_0$ weakly in $H^1(\RN)$ provided that $(z_n)_n$ is unbounded. 
Therefore, we can prove Claim similarly as \cref{lemma:3.3}. 
\end{proof}

Now we shall complete the proof. 
We first treat the case $V(x) \not \equiv V_\infty$. 
Let $t_0 \in (0,1)$ satisfy $\dis J_\infty(\gamma_\infty (t_0)) = \max_{0 \leq t \leq 1} J_\infty(\gamma_\infty (t))$. 
Since $\omega_0(x) > 0$ and $V(x) \not \equiv V_\infty$, \cref{lemma:4.1} gives 
	\[
		\begin{aligned}
			\limsup_{n \to \infty} J_{\lambda_n} (v_n) \leq \max_{ 0 \leq t \leq 1} J_\infty ( \gamma_\infty (t) ) 
			= J_\infty( \gamma_\infty (t_0) ) 
			&< J_{\infty, \infty} ( \gamma_\infty (t_0) ) 
			\\
			&\leq \max_{0 \leq t \leq 1} J_{\infty,\infty} (\gamma_\infty (t)) 
			= J_{\infty, \infty} (\omega_0).
		\end{aligned}
	\]
Noting that $J_{\infty,\infty} (\omega_0 ) > 0$ and $J_\infty(v) \geq 0$ provided that $J_\infty'(v) = 0$, 
we see from Claim that 
$k = 0$ and $\| v_n - v_0 \|_{H^1(\RN)} \to 0$. 
If $v_0 \equiv 0$, then \eqref{eq:4.2}, \cref{lemma:2.1} and \cref{lemma:4.1} lead to a contradiction as follows:
	\[
		0 < \liminf_{\lambda \to \infty} \lambda_n^{2/(p-2)} c_{\lambda_n}  
		= \liminf_{\lambda \to \infty} J_{\lambda_n} (v_n) 
		\leq \liminf_{\lambda \to \infty} 
		\int_{ \RN} \frac{1}{2} \left( |\nabla v_n|^2 + V(x) v_n^2 \right) - \frac{C_1}{p} v_n^p \, \rd x= 0.
	\]
Hence, $v_0 \not \equiv 0$ and we complete the proof in the case $V(x) \not \equiv V_{\infty}$. 

Finally, we assume that $V(x) \equiv V_\infty$. Then $J_{\infty} (v) = J_{\infty,\infty} (v)$ in this case. 
Hence, in Claim, it follows that either $v_0 \equiv 0$ or $v_0 = \omega_0$. 
In a similar way to the case $V(x) \not \equiv V_{\infty}$, 
we can prove that the case $k=0$ and $v_0=0$ does not occur. 
Therefore, either $k=0$, $v_0 = \omega_0$ or $k=1$, $v_0=0$, which implies 
$\| v_{n} (\cdot + z_n) - \omega_0 \|_{H^1(\RN)} \to 0$ for some $(z_n)_n \subset \RN$. 
Summarizing, we have proved that 
if $\lambda_n \to \infty$, then by taking a subsequence if necessary, 
$\| v_n (\cdot + z_n) - \omega_0 \|_{H^1(\RN)} \to 0$ for some $(z_n)_n \subset \RN$. 
Since $\omega_0$ is independent of choices of subsequences and 
$(\lambda_n)_n$ is arbitrary, we get 
$\| v_\lambda (\cdot + z_\lambda) - \omega_0 \|_{H^1(\RN)} \to 0$ as $\lambda \to \infty$ 
for some $(z_\lambda)_{\lambda \geq1} \subset \RN$. 
Hence, we complete the proof. 
	\end{proof}

	When we assume (V4) instead of (V2) and (V3), 
we first remark that 
$J_\infty \in C^1(H_V,\R)$ and 
it is not difficult to show the existence of critical point $w_0$ of $J_\infty$ 
corresponding to the mountain pass value of $J_\infty$. 
In this case, we use $\wt{\gamma}_\infty(t) := T_0 t w_0 \in C([0,1] , H_V)$ instead of $\gamma_\infty (t)$ 
where $T_0>1$ is chosen so that $J_\infty( \wt{\gamma}_\infty (1) ) < 0$, and we do not need $J_{\infty,\infty}$. 
Then we may prove \cref{lemma:4.1} and let $v_n \rightharpoonup v_\infty$ weakly in $H_V$ 
and $v_n(x) \to v_\infty(x)$ a.e. $x \in \RN$. 
Recalling \cref{lemma:3.7} and noting that $v_n$ satisfies 
$-\Delta v_n + V(x) v_n = \lambda_n^{(p-1)/(p-2)} g_{m_0} (u_{\lambda_n}  )$ in $\RN$ and 
$ \lambda_n^{ (p-1)/(p-2) }g_{m_0} \left( u_{\lambda_n}\right) \to v_\infty^{p-1}$ 
strongly in $H_V^\ast$, 
we observe that $v_n \to v_\infty$ strongly in $H_V$ and $v_\infty$ is a solution of 
$-\Delta v_\infty + V(x) v_\infty = v_\infty^{p-1}$ in $\RN$. 
Moreover, by $0< \lim_{n \to \infty} J_{\lambda_n} (v_n) = J_\infty(v_\infty)$, 
$v_\infty$ is nontrivial. 
Hence, \cref{theorem:1.3} also holds in this case.

	Finally, when $\Omega \subset \RN$ is bounded, 
the argument is similar to the above and we may verify that 
\cref{theorem:1.3} also holds in this case.

\section{$\mathcal{G}$-invariant case: proof of \cref{theorem:1.5}}\label{section:5}

In this section, we study \eqref{Plam} when $V(x)$ is not of trapping type. 
In this case, the main difficulty is to recover the compactness of bounded Palais-Smale sequences
because we are not able to argue as in \cref{proposition:3.4}.

	To overcome this difficulty, we assume that $\mathcal{G} \subset O(N)$ and $V(x)$ satisfy \eqref{g1} and (V5), 
and define 
\[
	H_\G^1(\R^N) := \{ u\in H^1(\R^N) \ | \ u(gx)=u(x) \ \hbox{for all} \ g\in \mathcal{G} \}.
\]
Notice that $I_\lambda$ is $\G$-invariant on $H^1(\RN)$, that is,
\[
I_{\lambda}(u(g \cdot )) =I_{\lambda}(u(\cdot)) \quad \hbox{for any} \ u\in H^1(\R^N) \ \hbox{and} \ g \in \G.
\]
By the principle of symmetric criticality due to Palais \cite{Pal} (see also \cite[Theorem 1.28]{Wi96}), 
we see that if the restriction $I_{\lambda}\vert_{H^1_\G(\R^N)}$ has a critical point, 
then it is in fact a critical point of $I_{\lambda}$. 
Thus it is sufficient to work in the restricted space $H^1_\G(\R^N)$.

	Next, due to the $\mathcal{G}$-symmetry, we have the following:

\begin{proposition} \label{proposition:5.1} 
Suppose that {\rm (V1)}, {\rm (V2)}, {\rm (V5)} and {\rm (f1)--(f4)} hold.
For $c< k_0 c_{\lambda,\infty}$ where $k_0 \geq 2$ is defined in \eqref{eq:1.4}, 
let $ (u_n)_n \subset H_\G^1(\R^N)$ be a sequence such that
\[
I_{\lambda}(u_n) \rightarrow c, \ I_{\lambda}'(u_n) \rightarrow 0 \ \hbox{in} \ (H_\G^1 (\RN) )^{\ast} 
\ \hbox{and} \ (\| u_n\|_{H^1(\RN)})_n \ \hbox{is bounded}.
\]
Then $(u_n)_n$ has a convergent subsequence in $H^1_{\G} (\RN)$.
\end{proposition}

	\begin{proof}
Let $(u_n)_n \subset H^1_{\mathcal{G}} (\RN)$ satisfy the assumption. 
We may find $u_0 \in H^1_{\mathcal{G}} (\RN)$, $k \geq 0$ and $\omega_j$ ($1 \leq j \leq k$) 
satisfying \cref{lemma:3.3}. 
From the $\mathcal{G}$-symmetry and \eqref{eq:1.4}, if $k > 0$, then $k \geq k_0$. 
Since we may show that $I_\lambda'(u) = 0$ implies $I_\lambda (u) \geq 0$, 
if $k > 0$, then 
	\[
		I_\lambda(u_n) \to c = I_\lambda(u_0) + \sum_{j=1}^{k} I_{\lambda,\infty} (\omega_j) 
		\geq k_0 c_{\lambda,\infty},
	\]
which is a contradiction. Hence, $k = 0$ and $(u_n)_n$ has a convergent subsequence in $H^1_{\G}(\RN)$. 
	\end{proof}

	To use \cref{proposition:5.1}, we next define the mountain pass value $c_{\lambda, \mathcal{G}}$ in $H^1_{\mathcal{G}} (\RN)$ by 
	\[
		\begin{aligned}
			c_{\lambda,\mathcal{G}} &:= \inf_{ \gamma \in \Gamma_{I_{\lambda}, \mathcal{G}} } 
			\max_{ 0 \leq t \leq 1} I_{\lambda} (\gamma(t)), 
			\\
			\Gamma_{I_{\lambda}, \mathcal{G}} &:= \Set{ \gamma \in C \left( [0,1], H_{\mathcal{G}}^1(\RN) \right)  |  
			\gamma(0)=0 \ \hbox{and} \ I_{\lambda}( \gamma(1)) <0 }. 
		\end{aligned}
	\]
Then, we may find a Palais--Smale--Cerami sequence $(u_{\lambda,n})$ at level $c_{\lambda, \mathcal{G}}$. 
Since $(u_{\lambda,n})_n$ is bounded in $H^1(\RN)$ due to \cref{proposition:3.2}, 
our next purpose is to show 
	\[
		c_{\lambda,\mathcal{G}} < k_0 c_{\lambda,\infty}.
	\]
To this end, we construct a suitable path $\gamma \in \Gamma_{I_{\lambda}, \mathcal{G}}$. 
Let $w_\lambda$ be a positive least energy solution of \eqref{eq:3.8} and 
recall $\{e_i\}_{1 \leq i \leq k_0}$ in \eqref{eq:1.5}. 
For $\ell \ge 1$ and $T>0$, our choice of the path is 
	\[
		\gamma_\ell (t) := t T \sum_{i=1}^{k_0} w_{\lambda}(x-\ell e_i) \in H^1_{\G} (\RN). 
	\]
Then arguing as in \cref{proposition:3.1}, 
one finds $T>0$ such that $\gamma_\ell \in \Gamma_{I_{\lambda}, \mathcal{G}}$ 
for all $ \ell \geq 1$.

\begin{proposition} \label{proposition:5.2} 
Suppose $N \ge 2$, \ef{g1}, {\rm (V1), (V2), (V5), (V6)}, {\rm (f1), (f2'), (f3), (f4)} and {\rm (f6)}.
Then for each $\lambda > 0$ there exists $\ell_0>0$ such that if $\ell \ge \ell_0$, 
\begin{equation} \label{eq:5.1}
	c_{\lambda, \mathcal{G}} \leq \max_{ 0 \leq t \leq 1} I_{\lambda} (\gamma_\ell(t)) 
	< k_0 c_{\lambda,\infty}.
\end{equation}
\end{proposition}

	To prove \cref{proposition:5.2}, some preliminaries are necessary. 
We first remark that there exist $0 < c_\lambda \leq C_\lambda < \infty$ such that 
	\begin{equation}\label{eq:5.2}
		c_\lambda \left( 1 + |x| \right)^{ -\frac{N-1}{2} } e^{ -\sqrt{V_\infty} |x| } 
		\leq w_\lambda (x) 
		\leq C_\lambda \left( 1 + |x| \right)^{ -\frac{N-1}{2} } e^{ -\sqrt{V_\infty} |x| }
		\quad \text{for every $x \in \RN$}.
	\end{equation}
In fact, by (f2') and $w_\lambda(x) \to 0$ as $|x| \to \infty$, 
we may prove that 
for all $\eta \in (0,\sqrt{V_\infty})$ there exists $C_\eta>0$ such that 
$w_\lambda (x) \leq C_\eta e^{ - \eta |x| }$ due to the comparison theorem 
(see the proof of \cref{proposition:6.1} below). 
Let $G$ be the Green function for $-\Delta + V_\infty$. Then 
	\[
		w_\lambda(x) = \int_{\RN} G(x-y) \lambda g_{m_0} \left( w_\lambda (y) \right) \rd y.
	\]
Using (f2') and arguing as in \cite[Proposition 4.1]{GNN81} or \cite[Proposition 1.2]{BL90}, we see that for some $C_\lambda>0$, 
	\[
		w_\lambda(x) \leq C_\lambda (1+|x|)^{-\frac{N-1}{2}} e^{ -\sqrt{V_\infty} |x| }.
	\]

	On the other hand, notice that $g_{m_0} (s) \geq 0$ and $-\Delta w_\lambda + V_\infty w_\lambda \geq 0$. 
Hence, by the comparison with $G$ on $\RN \setminus B_1(0)$, 
there exists $c_\lambda>0$ such that $c_\lambda G (x) \leq w_\lambda(x)$ for each $x \in \RN \setminus B_1(0)$. 
Therefore, \eqref{eq:5.2} holds.

	Next, for notational simplicity, we write $u_i(x)=w_{\lambda}(x-\ell e_i)$ for $i=1,\cdots,k_0$. 
By \cref{lemma:2.1} and (f6), 
$\frac{g_{m_0}(s)}{s}$ is nondecreasing on $(0,\infty)$. 
Therefore, we are able to apply the following lemma whose proof can be found in 
\cite[Lemma 4.4]{Hi08}.

\begin{lemma} \label{lemma:5.3}
Assume {\rm (f1), (f2'), (f6)} and
let $k$ be any integer with $k \ge 2$.

	\begin{enumerate}
	\item[{\rm (i)}] 
		If $u_i\in [0, \infty)$ for all $i=1,\cdots,k$, 
		then
		\[
		G_{m_0}\left(\sum_{i=1}^{k} u_i\right)-\sum_{i=1}^{k} G_{m_0}(u_i)
		-\frac{1}{2} \sum_{i,j=1,\ i\ne j}^{k} g_{m_0}(u_i)u_j\ge 0.
		\]
	\item[{\rm (ii)}]
		For any compact subset $K \subset (0,\infty)$, 
		there exist $C_K>0$ and $s_{K}>0$ such that 
		if $u_i\in K$ and $u_j\in(0,s_{K}]$ for each $j$ with $ j \neq i$, 
		then
		\begin{align*}
			G_{m_0}\left(\sum_{i=1}^{k} u_i\right)-
			&\sum_{i=1}^{k}G_{m_0}(u_i)-\frac{1}{2} \sum_{i,j=1,\ i\ne j}^{k} g_{m_0}(u_i)u_j 
			\geq C_K \sum_{j=1,\ j\not=i}^{k} u_j.
		\end{align*}
	\end{enumerate}
\end{lemma}

Next we use the following interaction estimates which can be shown 
from \eqref{eq:5.2} and \cite[Proposition 1.2]{BL90}
(see also \cite{Ada1, Hi07, Hi08, AW1}):

\begin{lemma} \label{lemma:5.4} 

Let $1 \leq i,j \leq k_0$ with $i \neq j$. 

	\begin{enumerate}
	\item[{\rm (i)}]
	For any $\nu>0$, there exists $C>0$ such that
		\[
			\intN u_i^{1+\nu} u_j \, \rd x 
			\le C e^{-\sqrt{V_{\infty}}|e_i-e_j|\ell} \ell^{-\frac{N-1}{2}}
			\quad \hbox{for all} \ \ell \ge 1.
		\]
	\item[{\rm (ii)}]
	There exist $C_1>0$ such that
		\[
			C_1 e^{-\sqrt{V_{\infty}}|e_i-e_j|\ell} \ell^{-\frac{N-1}{2}}
			\le \int_{B_1(\ell e_i)} u_j \, \rd x 
			\quad \hbox{for all} \ \ell \ge 1.
		\]
	\end{enumerate} 
\end{lemma}

	\begin{remark}\label{remark:5.5}
		Since $w_{\lambda}$ is a least energy solution of \ef{eq:3.8} and the nonlinear term in \ef{eq:3.8} depends on $\lambda$,
		constants in \cref{lemma:5.4} may also depend on $\lambda$. 
		When $N \geq 2$, this fact does not affect the proof of \cref{proposition:5.2},
		however, when $N=1$, this causes a serious problem as we will observe in the end of this section.
	\end{remark}

Under these preparations, we are ready to prove \cref{proposition:5.2}.
Although its proof is essentially same as in \cite{Ada1, Hi07, AW1},
we give the proof for the sake of completeness.

\begin{proof}[Proof of \cref{proposition:5.2}]
Let $\lambda>0$ be arbitrarily fixed. 
For every $s \in (0,\infty)$ and $v \in H^1_{\mathcal{G}} (\RN)$ with $v > 0$ in $\RN$, notice that 
	\begin{equation}\label{eq:5.3}
		\frac{I_\lambda'(s v) v}{s} = \| \nabla v \|_2^2 + \int_{\RN} V(x) v^2 \, \rd x 
		- \lambda \int_{\RN} \frac{g_{m_0} (sv) }{sv} v^2 \, \rd x
	\end{equation}
Thus, by $I_\lambda ( \gamma_\ell (1) ) < 0$ and the fact that $\frac{g_{m_0} (s) }{s}$ is nondecreasing in $(0,\infty)$, 
the function $[0,1] \ni t \mapsto I_\lambda (\gamma_\ell (t)) \in \R$ attains its maximum on $[t_{\ell,1}, t_{\ell,2}] \subset (0,1)$.  
In addition, by \eqref{eq:5.3} and $s^{-1} g_{m_0} (s) \to 0$ as $s \to 0^+$ thanks to (f2'), we have 
$t_{\ell,1} = t_{\ell,2} =:t_\ell $ and 
the maximum point is unique. This holds true for $[0,1] \ni t \mapsto I_{\lambda,\infty} (T t w_\lambda )$ 
and the maximum is attained only at $t=1/T$ since $I_{\lambda,\infty}'(w_\lambda) = 0$. 
From $I_\lambda (\gamma_\ell (t)) \to k_0 I_{\lambda,\infty} (tTw_\lambda)$ uniformly in $t \in [0,1]$ as $\ell \to \infty$, we see that 
	\[
		\wt{t}_{\ell} := t_{\ell} T \to 1 \quad \text{as $\ell \to \infty$}.
	\]
Therefore, in order to prove \ef{eq:5.1}, it suffices to show that 
for sufficiently large $\ell \ge 1$,
	\[
		I_{\lambda}(\gamma_\ell(t_{\ell}))< k_0 c_{\lambda, \infty}.
	\]
We divide its proof into 4 steps.

\medskip
\noindent
{\bf Step 1}: [Decomposition of the energy]. 

\medskip

	By direct calculations, we find that
	\[
	\begin{aligned}
	& I_{\lambda}(\gamma_\ell(t_{\ell})) 
	\\
	= \ & I_{\lambda, \infty} \left( \wt{t}_{\ell} \sum_{i=1}^{k_0} u_i \right)
	+\frac{\wt{t}_{\ell}^2}{2} \intN (V(x)-V_{\infty}) \left( \sum_{i=1}^{k_0} u_i \right)^2 \, \rd x \\
	= \ & \sum_{i=1}^{k_0} I_{\lambda, \infty} \left( \wt{t}_{\ell} u_i \right) 
	+\frac{\wt{t}_{\ell}^2}{2} \intN (V(x)-V_{\infty}) \left( \sum_{i=1}^{k_0} u_i \right)^2 \,\rd x 
	\\
	&
	+ \frac{\wt{t}_{\ell}^2}{2} \sum_{i,j=1, i \ne j}^{k_0} 
	\intN \nabla u_i \cdot \nabla u_j + V_{\infty} u_i u_j \, \rd x 
	- \lambda \intN \left\{ G_{m_0} \left( \wt{t}_{\ell} \sum_{i=1}^{k_0} u_i \right)
	-\sum_{i=1}^{k_0} G_{m_0} \left( \wt{t}_{\ell}  u_i \right) \right\} \rd x.
	\end{aligned}
	\]
Since $u_i(x)=w_{\lambda}(x- \ell e_i)$ are critical points of $I_{\lambda, \infty}$, it follows that
	\[
		\sum_{i,j=1, i \neq j}^{k_0} 
		\int_{ \RN} \nabla u_i \cdot \nabla u_j + V_\infty u_i u_j \, \rd x 
		= \sum_{i,j=1, i \neq j}^{k_0} \lambda \int_{ \RN} g_{m_0} (u_i) u_j \, \rd x .
	\]
Recalling $\displaystyle \max_{0 \leq t \leq 1} I_{\lambda,\infty} (t T w_\lambda) = c_{\lambda,\infty}$, we get 
	\[
	\begin{aligned}
	I_{\lambda} \left( \gamma_\ell  (t_{\ell}) \right) &\le 
	k_0 c_{\lambda,\infty}  
	+\frac{\wt{t}^2_{\ell}}{2} \intN (V(x)-V_{\infty}) \left( \sum_{i=1}^{k_0} u_i \right)^2 \, \rd x \\
	&\quad - \lambda \intN \left\{ G_{m_0} \left(  \sum_{i=1}^{k_0} \wt{t}_{\ell}  u_i\right)
	-\sum_{i=1}^{k_0} G_{m_0} \left( \wt{t}_{\ell} u_i \right) 
	- \frac{1}{2} \sum_{i,j=1, i \ne j}^{k_0} g_{m_0} \left( \wt{t}_{\ell} u_i \right) \wt{t}_{\ell} u_j \right\} \rd x \\
	&\quad + \frac{\lambda}{2} \sum_{i,j=1, i \ne j}^{k_0} \intN\left\{ 
	\wt{t}_{\ell}^2 g_{m_0}(u_i) u_j - g_{m_0} \left( \wt{t}_{\ell} u_i \right) \wt{t}_{\ell} u_j \right\} \rd x\\
	&=: {k_0} c_{\lambda,\infty} 
	+\frac{\wt{t}_{\ell}^2}{2} \intN (V(x)-V_{\infty}) \left( \sum_{i=1}^{k_0} u_i \right)^2 \,\rd x 
	-L_1(\R^N) + L_2(\R^N).
	\end{aligned}
	\]
By (V6), it holds that 
\[
\intN (V(x)-V_{\infty}) w_{\lambda}(x-\ell e_i)^2 \, \rd x
\le \kappa \intN e^{-\alpha |x|} w_{\lambda}(x-\ell e_i)^2 \,\rd x.
\]
Moreover in (V6) without loss of generality we may assume 
$\alpha \in (\alpha_0 \sqrt{V_{\infty}}, 2 \sqrt{V_{\infty}})$ if $\alpha_0 \in (0,2)$
and $\alpha \in (2 \sqrt{V_{\infty}}, \infty)$ if $\alpha_0=2$.
Then we are able to apply \cite[Proposition 1.2]{BL90} 
(cf. \cite[Lemma 2.4]{Ada1}, \cite[(4.5)]{Hi07}) to deduce that 
\[
\intN (V(x) -V_{\infty}) \left( \sum_{i=1}^{k_0} u_i \right)^2  \rd x 
\le C_0 \max \left\{ e^{-\alpha \ell},e^{-2 \sqrt{V_{\infty}} \ell}\ell^{-(N-1)} \right\} 
\quad \hbox{for all} \ \ell \ge 1.
\]
This yields that
\begin{equation} \label{eq:5.4}
I_{\lambda}(\gamma_\ell(t_{\ell}))\le k_0 c_{\lambda,\infty}
+\frac{C_0 \wt{t}_\ell^2}{2} \max \left\{ e^{-\alpha \ell},e^{-2 \sqrt{V_{\infty}} \ell}\ell^{-(N-1)}\right\} 
-L_1(\R^N)+L_2(\R^N).
\end{equation}

\medskip
\noindent
{\bf Step 2}: [Estimate of $L_1(\R^N)$]. 

\medskip

	Choose $\wt{\ell} = \wt{\ell} (\mathcal{G}) \geq 1$ 
so that $B_1(\ell e_i) \cap B_1(\ell e_j) =\emptyset$
for $i \ne j$ and $\ell \ge \wt{\ell}$.
We decompose
	\[
		\R^N=B_1(\ell e_1) \cup B_1(\ell e_2) \cup
		\cdots \cup B_1(\ell e_{k_0}) \cup \Omega, \quad 
		\Omega := \R^N \setminus \bigcup_{i=1}^{k_0} B_1(\ell e_i)
	\]
and write $B_i=B_1(\ell e_i)$. 
Then one has
	\[
		L_1(\R^N)=L_1(B_1)+L_1(B_2)
		+\cdots + L_1(B_{k_0})+L_1(\Omega).
	\]

	Next, set 
	\[
		K := \ov{\bigcup_{ \ell \geq 1 } \Set{ \wt{t}_{\ell} w_\lambda (x) | x \in B_1(0) } }.
	\]
By $\wt{t}_{\ell} \to 1$ as $\ell \to \infty$, we see that 
$K \subset (0,\infty)$ is compact. Furthermore, for each $i$ and $j$ with $i \neq j$, 
we may suppose $\wt{t}_{\ell} u_j(x) \leq s_K$ on $B_i$ 
provided $\ell \geq \wt{\ell}$ where $s_K>0$ is the constant in \cref{lemma:5.3} (ii). 
Then by \cref{lemma:5.3} (ii) and \cref{lemma:5.4} (ii), we get
	\[
\begin{aligned}
L_1(B_i)
&= \lambda \int_{B_i}
\left\{ G_{m_0} \left( \sum_{i=1}^{k_0} \wt{t}_{\ell} u_i \right) 
-\sum_{i=1}^{k_0} G_{m_0} \left( \wt{t}_{\ell} u_i \right) 
- \frac{1}{2} \sum_{i,j=1,i \ne j}^{k_0} g_{m_0} \left( \wt{t}_\ell u_i\right) \wt{t}_{\ell} u_j \right\} \rd x \\
&\ge \lambda C_K \wt{t}_{\ell} \sum_{j=1,j \ne i}^{k_0} \int_{B_i} u_j  \,\rd x \\
&\ge \lambda C_K C_1 \wt{t}_{\ell} \sum_{j=1,j \ne i}^{k_0}
e^{-\sqrt{V_{\infty}} |e_i-e_j| \ell} \ell^{-\frac{N-1}{2}}
\end{aligned}
	\]
and 
	\[
		\sum_{i=1}^{k_0} L_1(B_i)
		\ge \lambda C_K C_1 \wt{t}_{\ell} \sum_{i,j=1,i \ne j}^{k_0}
		e^{-\sqrt{V_{\infty}} |e_i-e_j| \ell} \ell^{-\frac{N-1}{2}}. 
	\]
Now \cref{lemma:5.3} (i) implies $L_1(\Omega) \ge 0$, from which we conclude that 
\begin{equation} \label{eq:5.5}
-L_1(\R^N) \le -\lambda C_K C_1 \wt{t}_{\ell} \sum_{i,j=1,i \ne j}^{k_0}
e^{-\sqrt{V_{\infty}}|e_i-e_j| \ell} \ell^{-\frac{N-1}{2}}
\quad \hbox{for} \ \ell \ge \tilde{\ell}.
\end{equation}

\medskip
\noindent
{\bf Step 3}: [Estimate of $L_2(\R^N)$]. 

\medskip

First we claim that for any $\varepsilon>0$, there exists $\wt{\ell}_\e \geq \wt{\ell}$ such that 
\begin{equation} \label{eq:5.6}
	\wt{t}_{\ell} g_{m_0}(w_{\lambda}(x)) - g_{m_0} \left( \wt{t}_{\ell} w_{\lambda}(x) \right) 
	\le \varepsilon w_{\lambda}(x)^{1+\nu}
	\quad \text{for all $x\in \R^N$ and $\ell \geq \wt{\ell}_\e$}.
\end{equation}
Indeed from (f2') and (f6), it follows that $g_{m_0}(s)\geq 0$ in $(0,\infty)$ and 
	\[
		\wt{t}_{\ell} g_{m_0}(w_{\lambda}(x)) - g_{m_0} \left( \wt{t}_{\ell} w_{\lambda}(x) \right) 
		\le 
		\wt{t}_{\ell} g(w_{\lambda}(x)) \le \varepsilon w_{\lambda}(x)^{1+\nu} 
		\quad \hbox{for} \ x \in \R^N \setminus
		B_{r_{\varepsilon}}(0) \ \hbox{and}\ \ell \geq \wt{\ell},
	\]
by taking sufficiently large $r_{\varepsilon}>0$. 
On the other hand since $\wt{t}_{\ell} g_{m_0}(w_{\lambda})-g_{m_0}( \wt{t}_{\ell} w_{\lambda}) \to 0$ as $\ell \to \infty$, 
there exist $\wt{\ell}_\e \geq \wt{\ell}$ such that
	\[
		\wt{t}_\ell g_{m_0}(w_{\lambda}(x)) -g_{m_0} \left( \wt{t}_\ell w_{\lambda}(x) \right) 
		\le \varepsilon w_{\lambda}(x)^{1+\nu}
		\quad \text{for any $ x \in B_{r_{\varepsilon}}(0)$ and $\ell \geq \wt{\ell}_\e$},
	\]
yielding \ef{eq:5.6}.

	From \ef{eq:5.6} and \cref{lemma:5.4} (i), for every $\ell \geq \wt{\ell}_\e$, we have
\begin{equation}\label{eq:5.7}
		L_2(\R^N) \le\frac{ \lambda \wt{t}_{\ell} \varepsilon}{2}
		\sum_{i,j=1, i \ne j}^{k_0} \intN u_i^{1+\nu} u_j \,dx 
		\le \lambda C \wt{t}_{\ell} \varepsilon \sum_{i,j=1,i \ne j}^{k_0}
		e^{-\sqrt{V_{\infty}} |e_i-e_j| \ell} \ell^{-\frac{N-1}{2}}.
\end{equation}

\medskip
\noindent
{\bf Step 4}: [Conclusion]. 

\medskip

	From \ef{eq:5.4}, \ef{eq:5.5} and \eqref{eq:5.7}, it holds that for $\ell \geq \wt{\ell}_\e$, 
	\begin{equation}\label{eq:5.8}
		\begin{aligned}
		c_{\lambda, \mathcal{G}} \leq I_{\lambda}(\gamma_\ell(t_\ell)) 
		&\le k_0 c_{\lambda,\infty}  
		+\frac{C_0 \wt{t}_\ell^2}{2} \max \left\{ 
		e^{-\alpha \ell}, e^{-2\sqrt{V_{\infty}}\ell} \ell^{-(N-1)} \right\} \\
		&\quad + \lambda \wt{t}_\ell \left(  C \e  - C_KC_1 \right) \sum_{i,j=1,i \ne j}^{k_0}
		e^{-\sqrt{V_{\infty}}|e_i-e_j|\ell} \ell^{-\frac{N-1}{2}} .
		\end{aligned}
	\end{equation}
By \eqref{eq:1.5}, choose $i_0 \neq j_0$ so that $|e_{i_0}-e_{j_0}| = \alpha_0$. 
Taking a small $\e>0$, we see from \eqref{eq:5.8} that 
for each $\ell \geq \wt{\ell}_\e$, 
	\[
		c_{\lambda, \mathcal{G}} 
		\leq k_0 c_{\lambda,\infty} 
		+\frac{C_0 \wt{t}_\ell^2}{2} \max \left\{ 
		e^{-\alpha \ell}, e^{-2\sqrt{V_{\infty}}\ell} \ell^{-(N-1)} \right\} 
		- \frac{\lambda \wt{t}_\ell C_K C_1 }{2} e^{ - \sqrt{V_\infty} \alpha_0 \ell } \ell^{ - \frac{N-1}{2} }. 
	\]
From $\alpha > \alpha_0 \sqrt{V_\infty}$ and $\wt{t}_\ell \to 1$ as $\ell \to \infty$, 
we may find $\ell_0$ such that if $\ell \geq \ell_0$, then 
	\[
		c_{\lambda, \mathcal{G}} \leq I_\lambda \left( \gamma_\ell (t_\ell) \right) 
		< k_0 c_{\lambda, \infty},
	\]
which completes the proof. 
\end{proof}

	Now we prove \cref{theorem:1.5}: 

\begin{proof}[Proof of \cref{theorem:1.5}]
By Propositions \ref{proposition:5.1} and \ref{proposition:5.2}, 
the auxiliary problem \ef{eq:3.1} admits a $\mathcal{G}$-invariant positive solution $u_{\lambda}$ 
corresponding to $c_{\lambda, \mathcal{G}}$. 
Since \cref{lemma:3.5} holds for $c_{\lambda,\infty}$ by choosing $V \equiv V_\infty$, 
from $c_{\lambda, \mathcal{G}} < k_0 c_{\lambda,\infty}$ and \cref{proposition:3.6}, 
we verify that $u_{\lambda}$ satisfies \eqref{Plam} 
for sufficiently large $\lambda$ and \cref{theorem:1.5} holds.
\end{proof}

	Finally, as noted in \cref{remark:5.5}, we explain the reason why we need the assumption $N \ge 2$ to prove \ef{eq:5.1}. 
When $N=1$, it follows that $\G = \set{\mathrm{id}, -\mathrm{id}}$ and $\alpha_0 = 2 = k_0$, and \eqref{eq:5.8} reduces to 
	\[
		I_{\lambda} \left( \gamma_\ell (t_\ell) \right) 
		\leq 2 c_{\lambda,\infty} 
		+ \left\{ \frac{C_0 \wt{t}_\ell^2}{2} + \lambda \wt{t}_\ell \left( C \e - C_K C_1 \right) \right\} e^{-2\sqrt{V_\infty} \ell}.
	\]
Notice that $C_0$, $C_K$ and $C_1$ are determined by $w_\lambda$ as mentioned in \cref{remark:5.5}, 
and hence, they depend on $\lambda$. 
Thus, it is not clear whether or not the following inequality holds:
	\[
		\frac{C_0 \wt{t}_\ell^2}{2} + \lambda \wt{t}_\ell \left( C \e - C_K C_1 \right) < 0.
	\]
From this reason, to exploit the same idea when $N=1$, 
we need to look at the dependence of $C_0,C_K,C_1$ on $\lambda$ more precisely. 
This will be an issue of next section.


\section{proof of \cref{theorem:1.6}}
\label{section:6}

	In this section, we suppose $N=1$, (V1), (V2), (V5'), (V6'), (f1), (f5) and (f6), and 
prove \cref{theorem:1.6}. To this end, as mentioned in the end of \cref{section:5}, 
we shall study the behavior of $(w_\lambda)$ as $\lambda \to \infty$,
where $w_{\lambda}$ is the least energy solution of \eqref{eq:3.8}.

	As in \cref{section:5}, we observe that for each $\lambda > 0$, $I_\lambda$ has the mountain pass geometry in 
$H^1_E(\R) := \Set{ u \in H^1(\R) | u(-x) = u(x) }$ and we can define $c_{\lambda,E} \in (0,\infty)$ by 
	\[
		c_{\lambda,E} := \inf_{ \gamma \in \Gamma_{I_\lambda,E} } \max_{ 0\leq t \leq 1} I_{\lambda} ( \gamma (t) ), \quad 
		\Gamma_{I_\lambda,E} := \Set{ \gamma \in C \big( [0,1], H^1_E(\R) \big) | 
		\gamma (0) = 0, \ I_{\lambda} (\gamma (1)) < 0 }.
	\]
Let $(u_{\lambda,n})_n \subset H^1_E(\R)$ be a Palais--Smale--Cerami sequence of $I_\lambda$ at level $c_{\lambda,E}$. 
Then $(u_{\lambda,n})_n$ is bounded in $H^1(\R)$ due to \cref{proposition:3.2}. 
Therefore, we aim to prove that there exists $\lambda_1 = \lambda_1(V_\infty,f) $ such that 
	\begin{equation}\label{eq:6.1}
		\lambda \geq \lambda_1 \quad \Rightarrow \quad 
		\text{$\exists u_\lambda \in H^1_E(\R), \ \exists (u_{\lambda,n_k})_k$ s.t. 
		$\| u_{\lambda,n_k} - u_{\lambda} \|_{H^1(\R)} \to 0$}.
	\end{equation}
Once \eqref{eq:6.1} is established, since we may prove $c_{\lambda, E} \leq M \lambda^{-2/(p-2)}$ for some $M>0$ 
as in \cref{lemma:3.5}, by \cref{proposition:3.6}, 
there exists $\lambda_0 = \lambda_0(V,f) \geq \lambda_1$ such that 
$u_{\lambda}$ is a solution of \eqref{Plam} provided $\lambda \geq \lambda_0$. Therefore \cref{theorem:1.6} holds. 

In what follows, our aim is to prove \eqref{eq:6.1}. 
For this purpose, 
recall $I_{\lambda,\infty}$ in \eqref{eq:3.9} and define the mountain pass value of $I_{\lambda,\infty}$ as in \eqref{eq:3.4}:
	\[
		c_{\lambda,\infty} := \inf_{\gamma \in \Gamma_{I_{\lambda,\infty}}} 
		\max_{0 \leq t \leq 1} I_{\lambda,\infty} (\gamma (t)), \ \ 
		\Gamma_{I_{\lambda,\infty}} := \Set{\gamma \in C \big( [0,1] , H^1(\R) \big) | \gamma (0) = 0, \ I_{\lambda,\infty} (\gamma (1)) < 0 }.
	\]
From the definition of $g_{m_0}$ and (f6), there exists $\lambda_2 = \lambda_2(V_\infty,f) > 0$ such that 
if $\lambda \geq \lambda_2$, then we may find $\wt{s}_\lambda  \in (0, s_0)$ so that 
	\[
		\lambda G_{m_0} (s) - \frac{V_\infty}{2} s^2 < 0 = \lambda G_{m_0} (\wt{s}_\lambda) - \frac{V_\infty}{2} (\wt{s}_\lambda)^2
		  \quad \text{for any $s \in (0,\wt{s}_\lambda)$}, \quad 
		\lambda g_{m_0} (\wt{s}_\lambda) - V_\infty \wt{s}_\lambda > 0.
	\]
Hence, when $\lambda \geq \lambda_2$, 
we can apply the result in \cite{JT03b} to find a unique positive solution $w_\lambda \in H^1_E(\R)$ such that 
	\[
		\left\{\begin{aligned}
			-w_\lambda'' + V_\infty w_\lambda &= \lambda g_{m_0} (w_\lambda) \quad \text{in} \ \R, 
			\quad w_\lambda (0) = \max_{x\in \R} w_\lambda(x), \quad  w_{\lambda}'(x) < 0 \quad \text{in} \ (0,\infty), \\
			w_{\lambda} (-x) &= w_{\lambda} (x) \quad \text{for all $x \in \R$}, \quad 
			I_{\lambda,\infty} ( w_\lambda) = c_{\lambda,\infty}.
		\end{aligned}\right.
	\]
Moreover, by \cref{theorem:1.3}, we have 
	\begin{equation}\label{eq:6.2}
		v_{\lambda} (x) := \lambda^{ 1/(p-2) } w_\lambda (x)
		\to \omega_0 (x) \quad \text{strongly in $H^1(\R)$} \ \hbox{as $\lambda \to \infty$},
	\end{equation}
where $\omega_0$ is a unique positive solution of 
	\[
		\left\{
		\begin{aligned}
			- \omega_0'' + V_\infty \omega_0 &= \omega_0^{p-1} \quad \text{in} \ \R, \quad 
			\omega_0(0) = \max_{x\in \R} \omega_0(x), \quad 
			\omega_0'(x) < 0 \quad \text{in} \ (0,\infty), 
			\\
			\omega_0 (-x) &= \omega_0 (x) \quad \text{for all $x \in \R$}.
		\end{aligned}
		\right.
	\]

	We first derive a uniform decay of $(v_\lambda)$ as $|x| \to \infty$:

\begin{proposition}\label{proposition:6.1}
There exist $\lambda_3 = \lambda_3(V_\infty,f) \geq \lambda_2$, $C_i =C_i(V_\infty,p) >0$ $(i=1,2)$ such that 
		if $\lambda \geq \lambda_3$, then 
			\[
				C_1 \exp \left( - \sqrt{V_\infty} \left|x\right| \right) \leq v_\lambda (x) 
				\leq C_2 \exp \left( - \sqrt{V_\infty} \left|x\right| \right) 
				\quad \text{for all $x \in \R$}.
			\]
	\end{proposition}

\begin{proof}
Since $v_\lambda (-x) = v_\lambda(x)$, it suffices to prove the inequalities in $[0,\infty)$. 
By the definition of $v_{\lambda}$, one has 
	\[
		- v_\lambda'' + V_\infty v_\lambda = \lambda^{ (p-1)/(p-2) } g_{m_0} \left( \lambda^{ - 1/(p-2) } v_\lambda \right) 
		\quad \text{in} \ \R.
	\]
From \eqref{eq:6.2} and (f5), we may find $\lambda_3 = \lambda_3(V_\infty, f) \geq \lambda_2$ such that 
if $\lambda \geq \lambda_3$, then 
	\begin{equation}\label{eq:6.3}
		g_{m_0} \left( \lambda^{- 1/(p-2)  } v_\lambda (x) \right) \leq 2 \left( \lambda^{ - 1/(p-2)  } v_\lambda(x) \right)^{p-1} 
		\quad \text{for every $x \in \R$},
	\end{equation}
and hence, 
	\begin{equation}\label{eq:6.4}
		-v_\lambda'' + V_\infty v_\lambda \leq 2 v_\lambda^{p-1} \quad \text{in} \ \R.
	\end{equation}

Next, we fix $\delta_{V_\infty,p} > 0$ and $R_{V_\infty,p} > 0$ so that 
	\begin{equation}\label{eq:6.5}
		(p-1) \sqrt{V_\infty -\delta_{V_\infty,p}} > \sqrt{V_\infty}, 
		\quad V_\infty - 2 \omega_0^{p-2} (R_{V_\infty,p}) \geq V_\infty - \frac{1}{2} \delta_{V_\infty,p}.
	\end{equation}
Noting \eqref{eq:6.2} and enlarging $\lambda_3$ if necessary, we may assume that 
	\[
		V_\infty - 2 v_\lambda^{p-2} (R_{V_\infty,p}) \geq V_\infty - \delta_{V_\infty,p} 
		\quad \text{for each $\lambda \geq \lambda_3$}.
	\]
Then by $v_\lambda'(x) < 0$ in $(0,\infty)$, one obtains 
	\[
		V_\infty - 2 v_\lambda^{p-2} (x) \geq V_\infty - \delta_{V_\infty,p} \quad 
		\text{for each $x \geq R_{V_\infty,p}$ and $\lambda \geq \lambda_3$}.
	\]
Combining this inequality with \eqref{eq:6.4}, we observe that for $\lambda \geq \lambda_3$,  
	\[
		-v_\lambda'' + \left( V_\infty - \delta_{V_\infty,p} \right) v_\lambda \leq 0 \quad \text{in} \ \left(R_{V_\infty,p},\infty\right).
	\]
Since $v_\lambda (R_{V_\infty,p}) \to \omega_0 (R_{V_\infty,p})$ as $\lambda \to \infty$, 
by taking $\lambda_3$ larger if necessary, we may suppose that 
there exists $C(V_\infty,p) > 0$ such that 
	\[
		v_\lambda \left( R_{V_\infty,p} \right) \leq C \exp \left( - \sqrt{V_\infty - \delta_{V_\infty,p}} \, R_{V_\infty,p} \right)
		\quad \text{provided that $\lambda \geq \lambda_3$.}
	\]
Noting that $v_\lambda(x) \to 0$ as $|x| \to \infty$ and that $W(x) := C \exp \left( - \sqrt{V_\infty - \delta_{V_\infty,p}} \, x \right)$ 
is a solution of $-w'' + (V_\infty - \delta_{V_\infty,p}) w = 0$ in $(R_{V_\infty,p},\infty)$, 
we see from the comparison theorem that 
	\[
		v_\lambda (x) \leq C \exp \left( - \sqrt{V_\infty -\delta_{V_\infty,p}} \, x \right) 
		\quad \text{for any $x \in (R_{V_\infty,p},\infty)$ and $\lambda \geq \lambda_3$}.
	\]
Thus by \eqref{eq:6.2}, we conclude that for some $\wt{C} = \wt{C} (V_\infty,p) > 0$, 
	\begin{equation}\label{eq:6.6}
		\lambda \geq \lambda_3 \quad \Rightarrow \quad 
		v_\lambda (x) \leq \wt{C} \exp \left( - \sqrt{V_\infty - \delta_{V_\infty,p}} \left|x\right| \right) \quad \text{for all $x \in \R$}.
	\end{equation}

Next, we recall that 
	\[
		G(x) := \frac{1}{2\sqrt{V_\infty}} \exp \left( - \sqrt{V_\infty} \left|x\right| \right)
	\]
is the Green function of $- \frac{\rd^2}{\rd x^2} + V_\infty$, from which we deduce that 
	\[
		v_\lambda (x) = \int_{\R} G(x-y) \lambda^{ (p-1)/(p-2)  } g_{m_0} \left(  \lambda^{ - 1/ (p-2)  } v_\lambda (y) \right) \rd y.
	\]
Thus, \eqref{eq:6.3} and \eqref{eq:6.6} yield 
	\[
		v_\lambda (x) \leq \int_{\R} G(x-y) 2 v_\lambda^{p-1} (y) \, \rd y 
		\leq 2 \wt{C}^{p-1} \int_{\R} G(x-y) \exp \left( - (p-1) \sqrt{V_\infty - \delta_{V_\infty,p}} \left|y\right| \right) \rd y.
	\]
By \eqref{eq:6.5} and \cite[Proposition 1.2]{BL90}, there exist $C=C(V_\infty,p)>0$ such that 
	\[
		\int_{\R} G(x-y) \exp \left( -(p-1) \sqrt{V_\infty - \delta_{V_\infty,p}} \left|y\right| \right) \rd y 
		\leq C \exp \left( - \sqrt{V_\infty} \left| x \right| \right) 
		\quad \text{for each $x \in \R$}.
	\]
Hence, for some $C_2=C_2(V_\infty,p)>0$, we have 
	\[
		\lambda \geq \lambda_3 \quad \Rightarrow \quad 
		v_\lambda (x) \leq C_2 \exp \left( - \sqrt{V_\infty} \left| x \right| \right) \quad \text{for all $x\in \R$}. 
	\]

For the lower estimate, since $g_{m_0} (s) \geq 0$ in $[0,\infty)$, we see that
	\[
		- v_\lambda '' + V_\infty v_\lambda \geq 0 \quad \text{in} \ \R.
	\]
By $v_\lambda (0) \to \omega_0 (0) > 0$ as $\lambda \to \infty$, by enlarging $\lambda_3$ if necessary, 
we find $C_1(V_\infty,p)>0$ so that 
	\[
		\lambda \geq \lambda_3 \quad \Rightarrow \quad 0 < C_1 \leq v_\lambda (0).
	\]
Applying the comparison theorem for $v_\lambda$ and $C_1 \exp ( - \sqrt{V_\infty} \, x )$ on $(0,\infty)$, 
we get 
	\[
		C_1 \exp \left( - \sqrt{V_\infty} \, x \right) \leq v_\lambda (x) 
		\quad \text{for each $x \in [0,\infty)$ and $\lambda \geq \lambda_3$}.
	\]
From $v_\lambda (-x) = v_\lambda (x)$, we complete the proof. 
	\end{proof}

In order to estimate the mountain pass value $c_{\lambda,E}$, we use the functional $J_{\infty,\infty} (u)$: 
	\[
		J_{\infty,\infty} (u) := \frac{1}{2} \int_{\R} (u')^2 + V_\infty u^2 \,\rd x - \frac{1}{p} \int_{\R} u_+^p \, \rd x.
	\]
We choose $T_0 > 0$ so that 
	\[
		J_{\infty,\infty} \left( T_0 \omega_0 \right) < -1
	\]
and set 
	\[
		v_{\lambda,n} (x) := v_\lambda (x+n), \ \ 
		w_{\lambda,n} (x) := w_{\lambda} (x+n), \ \ 
		\omega_{0,n} (x) := \omega_0 (x+n), \ \
		\gamma_{\lambda,n} (t) := t \left( w_{\lambda,n} + w_{\lambda,-n} \right).
	\]
Since \eqref{eq:6.2} and (f5) imply that as $\lambda \to \infty$, 
	\[
		\begin{aligned}
			\lambda^{ 2/(p-2) } I_{\lambda, \infty} \big( \gamma_{\lambda,n} \left( T_0 t \right) \big) 
			&= \lambda^{ 2/(p-2) } I_{\lambda, \infty} \left( \lambda^{-1/(p-2)} \lambda^{ 1/(p-2) } \gamma_{\lambda,n} \left( T_0 t \right) \right) 
			\\
			&\to J_{\infty,\infty} \big( t T_0 \left( \omega_{0,n} + \omega_{0,-n} \right) \big) 
			\quad \text{uniformly in $t \in [0,1]$ and $n \in \N$}
		\end{aligned}
	\]
and $J_{\infty,\infty} \big( t T_0( \omega_{0,n} + \omega_{0,-n} ) \big) \to 2 J_{\infty,\infty} (t T_0 \omega_0)$ as $n \to \infty$,  
we may find $n_0 \in \N$ and $\lambda_4 = \lambda_4(V_\infty,f) \geq \lambda_3$ such that 
	\[
		\lambda \geq \lambda_4, \ n \geq n_0 \quad \Rightarrow \quad 
		\lambda^{ 2/(p-2) } I_{\lambda,\infty} \big( \gamma_{\lambda,n} (T_0) \big) < -1.
	\]
We also notice that for each $ \lambda \geq \lambda_4$, 
	\[
		\lim_{n \to \infty}  \left| \lambda^{2/(p-2)} I_\lambda \big( \gamma_{\lambda,n} (T_0) \big) 
		- \lambda^{2/(p-2)} I_{\lambda,\infty} \big( \gamma_{\lambda,n} (T_0) \big) \right| = 0
		\quad \text{uniformly in $\lambda \geq \lambda_4$}.
	\]
Therefore, by enlarging $n_0$ if necessary, we have for $n \geq n_0$ and $\lambda \geq \lambda_4$, 
	\[
		\gamma_{\lambda,n} \left( T_0 \cdot \right) \in \Gamma_{I_\lambda,E}, \quad 
		c_{\lambda,E} \leq \max_{0 \leq t \leq 1} I_\lambda \big( \gamma_{\lambda,n} \left( T_0 t \right) \big).
	\]
Our next aim is to prove the following.
 
\begin{proposition}\label{proposition:6.2}
Let $\kappa > 0$ and $\alpha \in (2\sqrt{V_\infty} , \infty )$ be the constants in \emph{(V6')}. 
Then there exist $\lambda_5 = \lambda_5(V_\infty,f) \geq \lambda_4$, 
$C_0 = C_0(V_\infty,p) > 0$ and $\wt{C}_0 = \wt{C}_0(V_\infty,p) > 0$ such that if 
			\begin{equation}\label{eq:6.7}
				\lambda \geq \lambda_5, \quad - C_0 + \frac{\wt{C}_0 \kappa \alpha }{\alpha^2 - 4V_\infty} < 0,
			\end{equation}
		then for sufficiently large $n \in \N$, it holds that
			\begin{equation}\label{eq:6.8}
				c_{\lambda,E} \leq \max_{0 \leq s \leq T_0} I_\lambda \big( \gamma_{\lambda,n} (s) \big) 
				< 2 I_{\lambda,\infty} \left( w_\lambda \right) = 2 c_{\lambda,\infty}.
			\end{equation}
	\end{proposition}

\begin{proof}
By \eqref{eq:6.2} and (f5), for some $\lambda_5 = \lambda_5 (V_\infty,f) \geq \lambda_4$, we may assume that 
for each $\lambda \geq \lambda_5$, 
	\begin{equation}\label{eq:6.9}
		\begin{aligned}
			&T_0 \max_{x \in \R} \big( w_{\lambda,n} (x) + w_{\lambda,-n} (x) \big) < s_0, \quad 
			4 \max_{x \in \R} w_{\lambda} (x) < s_0,
			\\
			&\frac{1}{2}s^{p-1} \leq f(s) \leq 2 s^{p-1} \quad \text{for each $s \in \left[ 0, 4 \| w_\lambda \|_{L^\infty(\R)} \right]$}.
		\end{aligned}
	\end{equation}
Thus, one has
	\[
		G_{m_0} \Big( s \big( w_{\lambda,n} (x) + w_{\lambda,-n} (x) \big) \Big) 
		= F \Big( s \big( w_{\lambda,n} (x) + w_{\lambda,-n} (x) \big) \Big) 
		\quad \text{for every $s \in [0,T_0]$ and $x \in \R$}.
	\]
By using
	\[
		\lambda \int_{\R} f( w_{\lambda,-n} ) w_{\lambda,n} \, \rd x 
		= 
		\int_{\R} w_{\lambda,n}' w_{\lambda,-n}' + V_\infty w_{\lambda,n} w_{\lambda,-n} \, \rd x 
		= \lambda \int_{\R} f( w_{\lambda,n}  ) w_{\lambda,-n} \, \rd x,
	\]
we can compute $I_{\lambda} ( \gamma_{\lambda,n} (s)  )$ as 
	\begin{equation}\label{eq:6.10}
		\begin{aligned}
			& 	I_\lambda \left( \gamma_{\lambda,n} (s) \right)
			\\
			= \ &
			I_{\lambda,\infty} \left( \gamma_{\lambda,n} (s) \right) 
			+ \frac{1}{2} \int_{\R} \left( V(x) - V_\infty \right) s^2 \left( w_{\lambda,n} + w_{\lambda,-n} \right)^2 \rd x
			\\
			= \ & I_{\lambda,\infty} \left( s w_{\lambda,n} \right) + I_{\lambda,\infty} \left( s w_{\lambda,-n} \right) 
			+ s^2 \int_{\R} w_{\lambda,n}' w_{\lambda,-n}' + V_\infty w_{\lambda,n} w_{\lambda,-n} \, \rd x 
			\\
			& - \lambda \int_{\R} 
			F \big( s \left( w_{\lambda,n} + w_{\lambda,-n} \right) \big) 
			- F \left( s w_{\lambda,n} \right) - F \left( s w_{\lambda,-n}  \right) \rd x \\
			&+ \frac{1}{2} \int_{\R} \left( V(x) - V_\infty \right) s^2 \left( w_{\lambda,n} + w_{\lambda,-n} \right)^2 \rd x
			\\
			= \ & 2 I_{\lambda,\infty} \left( s w_{\lambda,n} \right) 
			\\
			& 
			- \lambda \int_{\R} 
			\biggl\{
			F \big( s \left( w_{\lambda,n} + w_{\lambda,-n} \right) \big) 
			- F \left( s w_{\lambda,n} \right) - F \left( s w_{\lambda,-n}  \right) 
			\\
			&\qquad\qquad
			- \frac{1}{2} f \left( s w_{\lambda,n} \right) s w_{\lambda,-n} - \frac{1}{2} f \left( s w_{\lambda,-n} \right) s w_{\lambda,n} 
			\biggr\} \rd x
			\\
			& + \frac{\lambda}{2} \int_{\R} 
			s^2 f( w_{\lambda,n} ) w_{\lambda,-n} + s^2 f(w_{\lambda,-n}) w_{\lambda,n} 
			- f( s w_{\lambda,n} ) s w_{\lambda,-n} - f( s w_{\lambda,-n} ) s w_{\lambda,n} \, \rd x
			\\
			& + \frac{1}{2} \int_{\R} \left( V(x) - V_\infty \right) s^2 \left( w_{\lambda,n} + w_{\lambda,-n} \right)^2 \rd x
			\\
			=: \ & 2 I_{\lambda,\infty} ( s w_{\lambda,n} ) 
			- L_1(n,s) + L_2(n,s) + L_3(n,s).
		\end{aligned}
	\end{equation}

Let $s_{\lambda,n} \in (0,T_0)$ satisfy 
	\[
		\max_{0 \leq s \leq T_0} I_{\lambda} \big( \gamma_{\lambda,n} (s) \big) 
		= I_\lambda \big( \gamma_{\lambda,n} \left( s_{\lambda,n} \right) \big).
	\]
We remark that 
	\[
		\lim_{n \to \infty} I_{\lambda} \left( \gamma_{\lambda,n} (s) \right) 
		= 2 I_{\lambda,\infty} \left( s w_{\lambda} \right) \quad \text{uniformly in $s \in [0,T_0]$}.
	\]
Furthermore, by the same reason as in the beginning of proof of \cref{proposition:5.2}, 
it follows from \eqref{eq:6.9} and (f6) that 
	\begin{equation}\label{eq:6.11}
		I_{\lambda,\infty} \left( s w_{\lambda} \right) < I_{\lambda, \infty} \left( w_\lambda \right) 
		\quad \text{for every $s \in [0,T_0] \setminus \{1\}$}. 
	\end{equation}
Therefore, for each $\lambda \geq \lambda_5$, we obtain 
	\begin{equation}\label{eq:6.12}
		s_{\lambda,n} \to 1 \quad \text{as $n \to \infty$}.
	\end{equation}

To prove \eqref{eq:6.8}, from \eqref{eq:6.10} and \eqref{eq:6.11}, it is enough to show that for sufficiently large $n$, 
	\begin{equation}\label{eq:6.13}
		\lambda^{ 2/(p-2) } \big\{ - L_1 \left( n, s_{\lambda,n} \right) 
		+ L_2 \left( n , s_{\lambda,n} \right) + L_3 \left( n, s_{\lambda,n} \right) \big\} 
		< 0.
	\end{equation}
We first treat $\lambda^{2/(p-2)} L_3(n, s_{\lambda,n} )$. 
By \eqref{eq:6.12}, we may assume $s_{\lambda,n} \leq 2$. 
It follows from (V6'), \cref{proposition:6.1} and $\alpha > 2 \sqrt{V_\infty}$ that 
	\begin{equation}\label{eq:6.14}
		\begin{aligned}
			&\lambda^{2/(p-2)} L_3 \left( n ,s_{\lambda,n} \right) 
			\\
			\leq \ & \frac{s_{\lambda,n}^2}{2} \kappa \int_{\R} \exp \left( - \alpha |x| \right) 
			\left( v_{\lambda,n} + v_{\lambda,-n} \right)^2 \rd x 
			\\
			\leq \ &4 \kappa \int_{\R} \exp \left( - \alpha |x| \right) \left(  v_{\lambda,n}^2 + v_{\lambda,-n}^2 \right) \rd x
			\\
			\leq \ & 8 C_2^2 \kappa
			\int_{0}^\infty \exp \left( - \alpha x \right) 
			\left\{ \exp \left( -2 \sqrt{V_\infty} \left| x- n \right| \right) + \exp \left( -2 \sqrt{V_\infty} \left( x + n \right) \right) \right\} \rd x
			\\
			= \ & 8 C_2^2 \kappa 
			\Biggl\{ \int_0^n \exp \left( - \alpha x - 2\sqrt{V_\infty} \left( n -x \right) \right) \rd x 
			\\
			&\qquad\qquad 
			+ \int_n^\infty \exp \left( - \left( \alpha + 2 \sqrt{V_\infty}   \right) x \right) \exp \left( 2\sqrt{V_\infty} \,n \right) \rd x 
			+ \frac{ \exp \left(  -2 \sqrt{V_\infty} \, n \right)}{\alpha + 2 \sqrt{V_\infty}} \Biggr\}
			\\
			= \ &
			8 C_2^2 \kappa \left[ \frac{ \exp \left(  - 2 \sqrt{V_\infty} \, n \right)}
			{\alpha - 2 \sqrt{V_\infty}} \left\{ 1 - \exp \left( - \left( \alpha - 2\sqrt{V_\infty} \right) n \right) \right\}
			+ \frac{\exp \left( -\alpha n\right)}
			{\alpha + 2 \sqrt{V_\infty}} + \frac{ \exp \left(  -2 \sqrt{V_\infty} \, n  \right)}{\alpha + 2 \sqrt{V_\infty}} \right]
			\\
			\leq \ & 
			8 C_2^2 \kappa \left\{ \frac{2\alpha}{\alpha^2 - 4 V_\infty} \exp \left( -2 \sqrt{V_\infty} \, n \right) 
			+ \frac{\exp \left( -\alpha n\right)}{\alpha + 2 \sqrt{V_\infty}}
			\right\}.
 		\end{aligned}
	\end{equation} 

Next, we consider $\lambda^{2/(p-2)} L_2(n, s_{\lambda,n} )$. 
By (f5) and \eqref{eq:6.12}, we have 
	\[
		\begin{aligned}
			s_{\lambda,n}^2 f \big( w_{\lambda, \pm n} (x) \big) 
			&= \left( 1 + o_{\lambda,n}(1) \right) \left( 1 + o_\lambda (1) \right) w_{\lambda,\pm n}^{p-1} (x), 
			\\
			s_{\lambda,n} f \left( s_{\lambda,n} w_{\lambda, \pm n} (x) \right) 
			&= \left( 1 + o_{\lambda,n}(1) \right) \left( 1 + o_\lambda (1) \right) 
			s_{\lambda,  n}^{p-1} w_{\lambda, \pm n}^{p-1} (x) 
			\\
			&= \left( 1 + o_{\lambda,n}(1) \right) \left( 1 + o_\lambda (1) \right) w_{\lambda, \pm n}^{p-1} (x)
		\end{aligned}
	\]
uniformly with respect to $x \in \R$ where for each fixed $\lambda$, 
$o_{\lambda,n} (1) \to 0$ as $n \to \infty$. 
Hence, \cref{proposition:6.1} and \cite[Proposition 1.2]{BL90} give
	\begin{equation}\label{eq:6.15}
		\begin{aligned}
			&\left| \lambda_n^{2/(p-2)} L_{2} \left( n, s_{\lambda,n} \right) \right| 
			\\
			\leq \ & \lambda_n^{p/(p-2)} 
			\left\{ o_{\lambda,n}(1) + o_\lambda (1) \right\} 
			\int_{\R} w_{\lambda,n}^{p-1} w_{\lambda,-n} + w_{\lambda,-n}^{p-1} w_{\lambda,n} \, \rd x
			\\ 
			\leq \ & \left\{ o_{\lambda,n}(1) + o_\lambda (1) \right\} 
			\int_{\R} \exp \left( - (p-1) \sqrt{V_\infty} \left| x -n \right| \right) \exp \left(  - \sqrt{V_\infty} \left| x + n \right|\right) 
			\rd x
			\\
			\leq \ & \left\{ o_{\lambda,n}(1) + o_\lambda (1) \right\} \exp \left( -2 \sqrt{V_\infty} \, n \right).
		\end{aligned}
	\end{equation}

	Finally, we consider $\lambda^{2/(p-2)} L_1(n, s_{\lambda,n} )$. 
By \eqref{eq:6.9}, (f5) and (f6), we can apply \cref{lemma:5.3} (i) to get 
	\begin{multline*}
		F \big( s_{\lambda,n} \left( w_{\lambda,n} + w_{\lambda,-n} \right) \big) 
		- F \left( s_{\lambda,n} w_{\lambda,n}  \right) - F \left( s_{\lambda,n} w_{\lambda,-n} \right) 
		\\
		- \frac{1}{2} \big\{ f \left( s_{\lambda,n} w_{\lambda,n} \right) s_{\lambda,n} w_{\lambda,-n}  
		+ f \left( s_{\lambda,n} w_{\lambda,-n} \right) s_{\lambda,n} w_{\lambda,n}  \big\} 
		\geq 0 \quad \text{in}  \ \R.
	\end{multline*}
Thus, one has
	\[
		\begin{aligned}
			&- \lambda^{2/(p-2)} L_1 \left( n, s_{\lambda,n} \right) 
			\\
			\leq \, & 
			-\lambda^{p/(p-2)} \int_{|x-n| \leq 1} 
			\Bigl[ F \left( s_{\lambda,n} \left( w_{\lambda,n} + w_{\lambda,-n} \right) \right) 
			- F \left( s_{\lambda,n} w_{\lambda,n} \right)  - F \left( s_{\lambda,n} w_{\lambda,-n} \right)  
			\\
			& \hspace{3.5cm}  - \frac{1}{2} \left\{  f \left( s_{\lambda,n} w_{\lambda,n} \right) s_{\lambda,n} w_{\lambda,-n} 
			+ f ( s_{\lambda,n} w_{\lambda,-n} ) s_{\lambda,n} w_{\lambda, n} \right\}  \Bigr] \, \rd x.
		\end{aligned}
	\]
Since we may assume $s_{\lambda,n} \leq 2$ by taking a large $n$, 
we see from \eqref{eq:6.9} that for each $x \in [n-1,n+1]$
	\[
		f \big( s_{\lambda,n} w_{\lambda,-n} (x) \big) \leq 2^p w_{\lambda,-n}^{p-1} (x), \quad 
		F \big( s_{\lambda,n} w_{\lambda,-n} (x) \big) 
		\leq \frac{2^{p+1}}{p} w_{\lambda,-n}^p(x). 
	\]
Hence, it holds that
	\begin{equation}\label{eq:6.16}
		\begin{aligned}
			& - \lambda^{2/(p-2)} L_1 \left( n, s_{\lambda,n} \right) 
			\\
			\leq &
			-\lambda^{p/(p-2)} 
			\int_{|x-n| \leq 1} 
			F \big( s_{\lambda,n} \left( w_{\lambda,n} + w_{\lambda,-n} \right) \big) 
			- F \left( s_{\lambda,n} w_{\lambda,n} \right)  - \frac{1}{2} f \left( s_{\lambda,n} w_{\lambda,n} \right) s_{\lambda,n} w_{\lambda,-n} \,  \rd x
			\\
			& + \int_{|x-n| \leq 1} 
			2^{p+1} \left( v_{\lambda,-n}^p + v_{\lambda,-n}^{p-1} v_{\lambda,n} \right) \rd x.
		\end{aligned}
	\end{equation}
For the last term, \cref{proposition:6.1} yields that 
	\begin{equation}\label{eq:6.17}
		\begin{aligned}
			&\int_{|x-n|\leq 1} 2^{p+1} \left( v_{\lambda,-n}^p + v_{\lambda,-n}^{p-1} v_{\lambda,n} \right) \rd x 
			\\
			\leq \ & 
			2^{p+1}C_2^{p} \int_{|x-n| \leq 1} 
			\left\{ \exp \left( - p \sqrt{V_\infty} \left| x + n \right|  \right) 
			+ \exp \left( - \left( p -1 \right) \sqrt{V_\infty} \left| x+n \right| \right) \right\} \rd x
			\\
			\leq \ & C \exp \left( - 2 \left( p-1 \right) \sqrt{V_\infty} \, n \right),
		\end{aligned}
	\end{equation}
for some $C=C(V_\infty,p) > 0$. 

On the other hand, by (f5), (f6) and \eqref{eq:6.9} with \eqref{eq:6.12}, for sufficiently large $n$, one has 
	\[
		\begin{aligned}
			&F \left( s_{\lambda,n} \big( w_{\lambda,n} + w_{\lambda,-n} \right) \big) - F \left( s_{\lambda,n} w_{\lambda,n} \right) 
			- \frac{1}{2} f \left( s_{\lambda,n} w_{\lambda,n} \right) s_{\lambda,n} w_{\lambda,-n}
			\\
			= \ & 
			\int_0^1 
			\left\{ f \left( s_{\lambda,n} w_{\lambda, n} + \theta s_{\lambda,n} w_{\lambda,-n} \right) 
			- \frac{1}{2} f \left( s_{\lambda,n} w_{\lambda,n} \right) 
			\right\} s_{\lambda,n} w_{\lambda,-n} \, \rd \theta 
			\\
			\geq \ & 
			\frac{1}{2} f \left( s_{\lambda,n} w_{\lambda,n} \right) s_{\lambda,n} w_{\lambda,-n} 
			\geq \frac{1}{4} \left( s_{\lambda,n} w_{\lambda,n} \right)^{p-1} s_{\lambda,n} w_{\lambda,-n}.
		\end{aligned}
	\]
Hence, \cref{proposition:6.1} and \eqref{eq:6.12} give 
	\begin{equation}\label{eq:6.18}
		\begin{aligned}
			&- \lambda^{p/(p-2)} \int_{|x-n| \leq 1} 
			F \big( s_{\lambda,n} \left( w_{\lambda,n} + w_{\lambda,-n} \right) \big) - F \left( s_{\lambda,n} w_{\lambda,n} \right) 
			- \frac{1}{2} f \left( s_{\lambda,n} w_{\lambda,n} \right) s_{\lambda,n} w_{\lambda,-n} \, \rd x
			\\
			\leq & - \frac{C_1^p}{4} \int_{|x-n| \leq 1}  \left( 1 + o_{\lambda,n}(1) \right) 
			\exp \left( - \left( p-1 \right) \sqrt{V_\infty} \left| x - n \right| \right) 
			\exp \left( - \sqrt{V_\infty} \left| x + n\right| \right) \rd x
			\\
			\leq & -c_1 \left( 1 + o_{\lambda,n} (1) \right) \exp \left( - 2 \sqrt{V_\infty} \, n \right)
		\end{aligned}
	\end{equation}
for some $c_1=c_1(V_\infty,p)>0$. From \eqref{eq:6.16}, \eqref{eq:6.17} and \eqref{eq:6.18}, it follows that 
	\begin{equation}\label{eq:6.19}
		- \lambda^{2/(p-2)} L_1 \left( n, s_{\lambda,n} \right) 
		\leq 
		- c_1 \exp \left( - 2 \sqrt{V_\infty} \, n \right) + C \exp \left( - 2 \left( p -1 \right) \sqrt{V_\infty} \, n \right).
	\end{equation}

Now by \eqref{eq:6.14}, \eqref{eq:6.15} and \eqref{eq:6.19}, we obtain 
	\begin{equation}\label{eq:6.20}
		\begin{aligned}
			&\lambda^{2/(p-2)} \left\{ - L_1 \left( n, s_{\lambda,n} \right) + L_2 \left( n, s_{\lambda,n} \right) 
			+ L_3 \left( n, s_{\lambda,n} \right) \right\} 
			\\
			\leq \ & 
			8 C_2^2 \kappa \left\{ \frac{2\alpha}{\alpha^2 - 4V_\infty} \exp \left( -2 \sqrt{V_\infty} \, n \right) 
			+ \frac{\exp \left( - \alpha n \right)}{\alpha + 2 \sqrt{V_\infty}}  \right\} 
			+ \left\{ o_{\lambda,n}(1) + o_\lambda (1) \right\} \exp \left( -2 \sqrt{V_\infty} \, n \right)
			\\
			& \quad - c_1 \left( 1 + o_{\lambda,n} (1) \right) \exp \left( - 2 \sqrt{V_\infty} \, n \right) 
			+ C \exp \left( -2  (p-1) \sqrt{V_\infty} \, n \right)
			\\
			= \ & 
			\left\{ \frac{16C_2^2 \kappa \alpha }{\alpha^2 - 4 V_\infty} + o_{\lambda,n}(1) + o_\lambda (1) - c_1 \right\} 
			\exp \left( - 2 \sqrt{V_\infty} \, n \right) 
			\\
			& \quad + C \left\{ \exp \left( - \alpha n \right) + \exp \left( -2  (p-1) \sqrt{V_\infty} \, n \right) \right\}.
		\end{aligned}
	\end{equation}
Choose $C_0$ and $\wt{C}_0$ in \eqref{eq:6.7} as
	\[
		C_0 \left( V_\infty,p \right) := c_1 (V_\infty,p), \quad 
		\wt{C}_0 \left( V_\infty, p \right) := 16 C_2^2.
	\]
Next, take a $\lambda_5(V_\infty,f)$ such that if $\lambda \geq \lambda_5$, then 
	\[
		\left| o_{\lambda}(1) \right| \leq \frac{1}{2} \left\{ C_0 -  \frac{\wt{C}_0\kappa \alpha}{\alpha^2 - 4V_\infty} \right\}.
	\]
After fixing $\lambda \geq \lambda_5$, by taking a large $n$ and noting that $\alpha > 2 \sqrt{V_\infty}$ and that $p-1 > 1$, 
we see from \eqref{eq:6.20} and \eqref{eq:6.7} that \eqref{eq:6.13} holds. 
Therefore, \eqref{eq:6.8} holds and we complete the proof. 
\end{proof}

	Now \cref{theorem:1.6} can be proved in a similar way to \cref{theorem:1.5} 
with the help of \cref{proposition:6.2}. 
Therefore, we omit the details.

	
%
%

\medskip
\subsection*{Acknowledgment}
This work was supported by JSPS KAKENHI Grant Numbers JP19K03590, JP19H01797, JP18K03362, JP21K03317
and by JSPS-NSFC joint research project \lq\lq Variational study of nonlinear PDEs" 
and by the Research Institute for Mathematical Sciences, an International Joint
Usage/Research Center located in Kyoto University.

\end{document}